\documentclass{amsart}

\usepackage{amsmath,amssymb,amscd}
\usepackage{graphicx}
\usepackage{wasysym}
\usepackage[dvips,all]{xy}
\usepackage{hyperref}
\input{quivers_KC.sty}
\input{appendix.sty}

\newtheorem{theorem}{Theorem}[section]
\newtheorem{lemma}[theorem]{Lemma}
\newtheorem{corollary}[theorem]{Corollary}
\newtheorem{proposition}[theorem]{Proposition}

\theoremstyle{definition}
\newtheorem{definition}[theorem]{Definition}
\newtheorem{example}[theorem]{Example}
\newtheorem{conjecture}[theorem]{Conjecture}

\theoremstyle{remark}
\newtheorem{remark}[theorem]{Remark}

\newtheorem{problem}[theorem]{Problem}

\numberwithin{equation}{section}

\newcommand{\op}[1]{\operatorname{#1}}
\newcommand{\GL}{\operatorname{GL}}
\newcommand{\SL}{\operatorname{SL}}
\newcommand{\Gr}{\operatorname{Gr}}

\newcommand{\module}{\operatorname{mod}}
\newcommand{\Rep}{\operatorname{Rep}}
\newcommand{\Mod}{\operatorname{Mod}}
\newcommand{\Hom}{\operatorname{Hom}}
\newcommand{\PHom}{\operatorname{PHom}}
\newcommand{\Aut}{\operatorname{Aut}}

\newcommand{\Ext}{\operatorname{Ext}}
\newcommand{\ext}{\operatorname{ext}}

\newcommand{\Coker}{\operatorname{Coker}}

\newcommand{\T}{\operatorname{T}}
\newcommand{\mb}[1]{\mathbb{#1}}
\newcommand{\mc}[1]{\mathcal{#1}}
\newcommand{\mf}[1]{\mathfrak{#1}}
\newcommand{\mr}[1]{{\sf #1}}
\newcommand{\bs}[1]{\boldsymbol{#1}}
\renewcommand{\b}[1]{\bold{#1}}
\newcommand{\e}{{\sf e}}
\newcommand{\f}{{\sf f}}
\newcommand{\g}{{\sf g}}
\newcommand{\h}{{\sf h}}
\renewcommand{\S}{{\bf S}}
\newcommand{\ep}{{\epsilon}}
\newcommand{\vep}{{\varepsilon}}
\newcommand{\SI}{\operatorname{SI}}

\newcommand{\Spec}{\operatorname{Spec}}

\newcommand{\proj}{\operatorname{proj}\text{-}}

\newcommand{\br}[1]{\overline{#1}}
\newcommand{\innerprod}[1]{\langle#1\rangle}

\newcommand{\sm}[1]{\left(\begin{smallmatrix}#1\end{smallmatrix}\right)}

\newcommand{\fr}[1]{\framebox[1.2\width]{{$#1$}}}
\renewcommand{\b}[1]{\bold{#1}}
\newcommand{\bl}{{\beta_l}}
\newcommand{\dimbar}{\underline{\dim}}

\newcommand{\wtd}[1]{\widetilde{#1}}
\newcommand{\cone}[1]{\mb{R}_+\Sigma_{\beta_{#1}}({K_{#1,#1}^2})}
\newcommand{\sgn}{\operatorname{sgn}}

\newcommand{\Tr}{\operatorname{Tr}}
\newcommand{\ckQ}{\widehat{k\Delta}}
\newcommand{\uca}{\br{\mc{C}}}

\begin{document}

\title{Cluster Algebras, Invariant Theory, and Kronecker Coefficients I}
\author{Jiarui Fei}
\address{National Center for Theoretical Sciences, Taipei 30043, Taiwan}
\email{jrfei@ncts.ntu.edu.tw}
\thanks{}


\date{}
\dedicatory{}

\begin{abstract} We relate the $m$-truncated Kronecker products of symmetric functions to the semi-invariant rings of a family of quiver representations.
We find cluster algebra structures for these semi-invariant rings when $m=2$.
Each $\g$-vector cone $\mr{G}_{\Diamond_l}$ of these cluster algebras controls the $2$-truncated Kronecker products for all symmetric functions of degree no greater than $l$.
As a consequence, each relevant Kronecker coefficient is the difference of the number of the lattice points inside two rational polytopes.
We also give explicit description of all $\mr{G}_{\Diamond_l}$'s.
As an application, we compute some invariant rings.
\end{abstract}

\maketitle

\section*{Introduction}
Given a partition $\lambda$ of $n$, let $\S^\lambda$ be the associated irreducible complex representation of the symmetric group $\mf{S}_n$.
The {\em Kronecker coefficients} $g_{\mu,\nu}^\lambda$ are the tensor product multiplicities:
$$\S^\mu \otimes \S^\nu \cong \bigoplus_\lambda g_{\mu,\nu}^\lambda \S^\lambda.$$
To determine these coefficients and understand their properties has been one of the major problems in combinatorics and representation theory for nearly a century.
People are particularly interested in finding combinatorial interpretation for these coefficients.
They hope that Kronecker coefficients count some combinatorial objects, eg., lattice points in polytopes.
More recently, the interest in computing Kronecker coefficients has intensified in connection
with {\em Geometric Complexity Theory} \cite{M}, pioneered as an approach to the P vs. NP problem.
Despite of a large body of work, those coefficients remain very mysterious.

The complete solution of analogous problems for the general linear groups $\GL_n$ certainly bring new hope to this old problem.
The tensor multiplicities of irreducible representations of $\GL_n$, known as {\em Littlewood-Richardson coefficients}, can be computed by the Littlewood-Richardson rule.
According to Knutson-Tao \cite{KT}, they are counted by the lattice points in the so-called hive polytopes.
Similar results were also obtained by Derksen-Weyman in \cite{DW1,DW2} via an invariant-theoretic approach.
The link between these two approachs was established in \cite{Fs1} through Fomin-Zelevinsky's cluster algebras \cite{FZ1,BFZ,FZ4} and their quiver with potential models \cite{DWZ1,DWZ2}.
In this paper, we use the similar ideas in \cite{DW1,Fs1} to study the Kronecker coefficients.

It may be quite frustrating to study the Kronecker coefficients via a purely finite-group theoretic approach. However, by the Schur-Weyl duality they can be interpreted as multiplicities below
\begin{equation*} \label{eq:KC} \S_\lambda(V\otimes W) = \sum_{\mu,\nu} g_{\mu,\nu}^\lambda \S_\mu(V)\otimes \S_\nu(W). \end{equation*}
Then calculation of Kronecker coefficients becomes a problem in invariant theory.
Analogous to \cite{DW1}, our first main result is make this link completely explicit via quiver representations.

Let $K_{l,l}^m$ be the {\em flagged Kronecker quiver}
$$\kronml{-1}{-(l-1)}{-l}{l}{l-1}{1}{}$$
and $\bl$ be the {dimension vector} defined by $\bl(i)=|i|$.
Let $V$ be a $\beta_l$-dimensional vector space and $W$ be the $m$-dimensional vector space spanned by the arrows from $-l$ to $l$.
Consider the product of special linear group
$$\SL_{\bl}:=\prod_{i}\SL(V_i)$$
acting naturally on the quiver representation space
$$\Rep_{\bl}(K_{l,l}^m):=\bigoplus_{i={1}}^{l-1} \left(\Hom(V_{-i},V_{-(i+1)}) \oplus \Hom(V_{i+1},V_{i}) \right) \oplus \Hom(V_{-l},V_{l})\otimes W.$$
The {\em semi-invariant ring} $\SI_\bl(K_{l,l}^m):= k[\Rep_{\bl}(K_{l,l}^m)]^{\SL_\bl}$
is graded by a weight $\sigma\in \mb{Z}^{2l}$ and a weight $\lambda$ of $\GL(W)$:
$$\SI_\bl(K_{l,l}^m)=\bigoplus_{\sigma,\lambda} \SI_\bl(K_{l,l}^m)_{\sigma,\lambda}.$$
For any pair of partitions $\mu$ and $\nu$ of {\em length} no greater than $l$, we can associate a weight vector $\sigma(\mu,\nu)$ (see Lemma \ref{L:SI(Kml)}).

\begin{theorem}[Corollary \ref{C:dimUinv} and \ref{C:SI(Kml)2}] \label{T:FKm}
Let $(\mu,\nu,\lambda)$ be a triple of partitions of length no greater than $l,l$ and $m$ respectively, then
$$g_{\mu,\nu}^\lambda = \dim \SI_\bl(K_{l,l}^m)_{\sigma(\mu,\nu),\lambda}^U=\sum_{\omega\in \mf{S}_m} \op{sgn}(\omega) \dim \SI_\bl(K_{l,l}^m)_{\sigma(\mu,\nu),\lambda^\omega},$$
where $U$ is the group of upper triangular matrices in $\GL(W)$, and $\lambda^\omega$ is the weight defined by $(\lambda^\omega)(i) = \lambda(i)-i+\omega(i)$.
\end{theorem}
\noindent It is well-known that the Kronecker coefficients are invariant under the permutation of $\lambda,\mu,$ and $\nu$.
In the literature, it is usually written as $g_{\lambda,\mu,\nu}$.
Unfortunately, this symmetry is blurred from our construction, so we keep the notation $g_{\mu,\nu}^\lambda$ throughout.

To study the Kronecker coefficients, we need more structure in these semi-invariant rings.
It turns out that when $m=2$ all these semi-invariant rings are {\em upper cluster algebras} introduced by Fomin-Zelevinsky and Berenstein \cite{BFZ}.
Here is the crucial result.

\begin{theorem}[Theorem \ref{T:CS} and \ref{T:LP_Gl}] \label{T:crucial}
The semi-invariant ring $\SI_\bl(K_{l,l}^{2})$ is the upper cluster algebra $\br{\mc{C}}(\Diamond_l,\mc{S}_l)$.
Moreover, the {\em generic cluster character} maps the lattice points in a rational polyhedral cone $\mr{G}_{\Diamond_l}$ bijectively onto a basis of $\br{\mc{C}}(\Diamond_l)$.
We have a concrete description of the cone $\mr{G}_{\Diamond_l}$.
\end{theorem}
\noindent Here, $\Diamond_l$ is the ice {\em diamond quiver} to be introduced in Section \ref{S:Diamond}.
When $m=4$, it is displayed in Figure \ref{F:Diamond4}.
The initial {\em cluster variables} together with the {\em coefficients} \cite{FZ4} in $\mc{S}_l$ can be explicitly described in terms of {\em Schofield's semi-invariants} \cite{S1}.
We want to point out here that all these upper cluster algebras strictly contain the corresponding cluster algebras ${\mc{C}}(\Diamond_l)$.
This might suggest that upper cluster algebras are more important than cluster algebras.

As mentioned above, we dream that Kronecker coefficients count lattice points in some polytopes.
With the help of Theorem \ref{T:crucial}, we are getting closer to this dream.
We realize the relevant Kronecker coefficients as the difference of the numbers of lattice points in two polytopes. All these polytopes are obtained from the cone $\mr{G}_{\Diamond_l}$ by hyperplane sections defined by the weights $\sigma$ and $\lambda$.
We denote such a polytope by $\mr{G}_{\Diamond_l}(\sigma,\lambda)$ (or $\mr{G}_{\Diamond_l}(\sigma)$ if no restriction on $\lambda$).

\begin{theorem}[Proposition \ref{P:2KP}, Theorem \ref{T:KC}] \label{T:second}
Let $\mu$ and $\nu$ be two partitions of length no greater than $l$, and $\sigma:=\sigma(\mu,\nu)$ be the associated weight. Then
\begin{enumerate}
\item $$
\left[ s_{\mu} * s_{\nu} \right]_2 = \sum_{\g \in \mr{G}_{\Diamond_l}(\sigma) \cap \mb{Z}^{\Diamond_l^0}} \b{z}^{\lambda(\g)}.
$$
\item Let $\lambda$ be a partition of length $\leq 2$, and $\lambda'=(\lambda(1)+1,\lambda(2)-1)$, then
$$g_{\mu,\nu}^\lambda = |\mr{G}_{\Diamond_l}(\sigma,\lambda)\cap \mb{Z}^{\Diamond_l^0}| - |\mr{G}_{\Diamond_l}(\sigma,\lambda')\cap \mb{Z}^{\Diamond_l^0}|.$$
\end{enumerate}
Here, $\left[ - * - \right]_2$ is the $2$-truncated Kronecker product of symmetric functions (see Section \ref{S:KC}),
and $\lambda(\g)$ is the $\lambda$-weight of $\g$ (see Section \ref{S:CCS}).
\end{theorem}

In this paper, we treat the case of $m=2$ only. There are several reason for this.
First, $m=2$ is the first nontrivial case. We will see that to prove these results is already quite difficult.
To author's best knowledge, there is no general result like this in the literature.
Second, we do find cluster structures for $m>2$, but there are many different cluster structures for fixed $l$ and $m$.
However, we believe that for $m=2$ the cluster structure is unique, though we do not have a proof for this.
Finally, we will see in the subsequential paper that the cluster structure established here is a part of the building blocks for bigger $m$.
So we think that this case deserves our special care.

Here is the outline of this paper.
In Section \ref{S:KC}, we recall some basic facts on the Kronecker coefficients.
In Section \ref{S:SI}, we recall the work of Schofield, Derksen-Weyman, etc., on the semi-invariant rings of quiver representations.
In Section \ref{S:FKm}, we introduce the flagged Kronecker quivers and prove Theorem \ref{T:FKm}.
We also introduce a special class of semi-invariants, which will be our choice of initial cluster.
In Section \ref{S:CA}, we recall the definition of graded cluster algebras and their upper bounds.
In Section \ref{S:CC}, we recall the ice quivers with potentials and their representations following \cite{DWZ1,Fs1}.
We recall the generic cluster character in the current setting following \cite{P,Fs1}.

In Section \ref{S:Diamond}, we introduce the ice diamond quiver with potential $(\Diamond_l,W_l)$.
We give in Theorem \ref{T:LP_Gl} a concrete description of the domain of the generic cluster character as the lattice points in the polyhedral cone $\mr{G}_{\Diamond_l}$.
In Section \ref{S:CS}, we partially prove our first main result - Theorem \ref{T:crucial}.
The full proof is quite complicated so we put off it to the last section.
In Section \ref{S:CCS}, we draw several consequences from the cluster structure.
We prove our second main result - Theorem \ref{T:second}.
As a minor result, in Proposition \ref{P:gUinv} we estimate all possible $\g$-vectors in the $U$-invariants.
In Section \ref{S:Inv}, we introduce the {\em extremal unimodular fan} as a tool to present upper cluster algebras. As an application, we compute several relevant invariant rings.
In Section \ref{S:proof}, we fully prove Theorem \ref{T:crucial}.

\subsection*{Notations and Conventions}
Our vectors are exclusively row vectors. All modules are right modules.
For a quiver $Q$, we denote by $Q_0$ the set of vertices and by $Q_1$ the set of arrows.
For an arrow $a$, we denote by $t(a)$ and $h(a)$ its tail and head.
Arrows are composed from left to right, i.e., $ab$ is the path $\cdot \xrightarrow{a}\cdot \xrightarrow{b} \cdot$.
Throughout the paper, the base field $k$ is algebraically closed of characteristic zero.
Unadorned $\Hom$ and $\otimes$ are all over the base field $k$, and the superscript $*$ is the trivial dual for vector spaces.
For any representation $M$, $\dimbar M$ is the dimension vector of $M$.
For direct sum of $n$ copies of $M$, we write $nM$ instead of the traditional $M^{\oplus n}$.


\section{Basics on Kronecker Coefficients} \label{S:KC}
A {\em partition} $\lambda$ is a weakly decreasing sequence $\left(\lambda(1),\lambda(2),\dots\right)$ eventually being zeros.
We often do not write the trailing zeros, but the zero partition $(0, 0, \cdots)$ is also valid.
The {\em length} of $\lambda$ denoted by $\ell(\lambda)$ is the number of indices $i$ for which $\lambda(i)$ is non-zero.
The {\em size} of $\lambda$ denoted by $|\lambda|$ is the sum of $\lambda(i)$'s.
If $\lambda$ has size $n$, we also say that $\lambda$ is a partition of $n$, and write $\lambda \vdash n$.
When $\lambda$ contains some repetition, we may use superscripts to simplify the notation.
For example, $(5,3^2,2^3,1)$ is the partition $(5,3,3,2,2,2,1)$.
The {\em transpose} or {\em conjugate} of $\lambda$ is the partition $\lambda^*$ defined by
$$\lambda^*(i) = \#\{j\mid \lambda(j)\geq i\}.$$

We fix an algebraically closed field of characteristic $0$ as our base field $k$.
For $n\in\mb{N}$ let $\mf{S}_n$ be the symmetric group of $n$ letters.
The irreducible representations of $\mf{S}_n$ over $k$ are parameterized by partitions of $n$.
For a partition $\lambda$, we denote by $\S^\lambda$ the irreducible representation corresponding to $\lambda$.
The category of all representations of $\mf{S}_n$ is semisimple.
The {\em Kronecker coefficients} $g_{\mu,\nu}^\lambda$ are the tensor product multiplicities:
$$\S^\mu \otimes \S^\nu \cong \bigoplus_\lambda g_{\mu,\nu}^\lambda \S^\lambda.$$
The {\em trivial representation} and the {\em sign representation} are the two one-dimensional representations corresponding to the partition $(n)$ and $(1^n)$.
They are characterized by the property that
\begin{equation*} \S^\lambda \otimes \S^{(n)} = \S^\lambda,\ \text{ and }\ \S^\lambda \otimes \S^{(1^n)} = \S^{\lambda^*}.
\end{equation*}
In particular, we have that $g_{\mu,\nu}^\lambda=g_{\mu^*,\nu^*}^\lambda$.
Since all representations of $\mf{S}_n$ are self-dual, $g_{\mu,\nu}^\lambda$ is invariant under the permutations of $\lambda,\mu$ and $\nu$.

The {\em Schur functor} $\S_\lambda$ is by definition
$$\S_\lambda(V) = \Hom_{\mf{S}_n} (\S^\lambda, V^{\otimes n}),$$
where $V$ is a $k$-vector space.
It follows from the definition that
\begin{lemma} \label{L:SW} We have the following decomposition
$$\S_\lambda(V\otimes W) = \bigoplus_{\mu,\nu} g_{\mu,\nu}^\lambda \S_\mu(V) \otimes \S_\nu(W).$$
\end{lemma}
\noindent By the {\em Schur-Weyl duality},
$\S_\lambda(V)$ is an irreducible {\em polynomial} representation of the general linear group $\GL(V)$ if $\dim V \geq \ell(\lambda)$.
It is well-known that the category of all polynomial representations of $\GL(V)$ is also semisimple,
and the irreducible representations are parameterized by partitions of length no greater than $\dim V$.
The tensor product multiplicities $c_{\mu,\nu}^\lambda$ are called {\em Littlewood-Richardson coefficients} (LR coefficient for short):
$$\S_\mu(V) \otimes \S_\nu(V) \cong \bigoplus_\lambda c_{\mu,\nu}^\lambda \S_\lambda(V).$$
The LR coefficients can be computed using the Littlewood-Richardson rule \cite{FH}.

We fix the dimension of the vector space $V$, say $\dim V=m$.
Let $\Lambda_n^m$ be the space of homogeneous symmetric polynomials of degree $n$ in $m$ variables $z_1,\dots,z_m$.
A basis for this space is given by the {\em Schur polynomials} $\{s_\lambda\mid \lambda\vdash n\}$.
Let $\Lambda^m:=\bigoplus_{n\geq 0} \Lambda_n^m$ be the algebra of symmetric polynomials.
The {\em character map} $\chi$ maps a representation $M$ of $\GL(V)$ to $\sum_{\lambda} \dim(M_\lambda) \b{z}^\lambda\in \Lambda^m$,
where $M_\lambda$ is subspace of $M$ of weight $\lambda$.
It is well-known \cite{FH} that $\chi(\S_\mu(V)\otimes \S_\nu(V)) = s_\mu s_\nu$.
In particular, the structure constants in $\Lambda^m$ is also given by the LR coefficients:
$$s_\mu s_\nu = \sum_\lambda c_{\mu,\nu}^\lambda s_\lambda.$$
The {\em $m$-truncated Kronecker product} in $\Lambda^m$ is by definition
$$\left[s_\mu * s_\nu\right]_m := \sum_{\lambda\mid \ell(\lambda)\leq m}g_{\mu,\nu}^\lambda s_\lambda.$$


As is usual in the theory of {\em symmetric functions}, we can use inverse limit (see \cite{Mc}) to obtain a {\em Schur function} $s_{\lambda}$ for each partition, depending on countably many variables.
We recall that the {\em Kronecker product} on the algebra of symmetric functions is defined by
$$s_\mu * s_\nu := \sum_{\lambda}g_{\mu,\nu}^\lambda s_\lambda.$$
The Kronecker product and the ordinary product satisfy the Littlewood's identity:
\begin{equation*}(s_\mu s_\nu)*s_\lambda = \sum_{\tau\vdash |\mu|, \eta \vdash |\nu|} c_{\tau,\eta}^{\lambda} (s_\tau * s_\mu)(s_\eta*s_\nu).
\end{equation*}

\noindent Using this identity and Jacobi-Trudi formula, we can express Kronecker coefficients in terms of multiple LR coefficients. For a set $\bs{\eta}$ of $m$ partitions $\{\eta_1,\eta_2,\dots,\eta_m\}$,
the {\em multiple} LR coefficients $c_{\bs{\eta}}^\lambda$ are the multiplicities of $\S_\lambda(V)$ in the multiple tensor product:
$$\S_{\eta_1}(V) \otimes \S_{\eta_2}(V)\otimes \cdots \otimes \S_{\eta_m}(V) \cong \bigoplus_\lambda c_{\bs{\eta}}^\lambda \S_\lambda(V).$$
For $\omega\in \mf{S}_m$, we write $\lambda^\omega$ for the vector in $\mb{Z}^m$ such that $\lambda^\omega(i)= \lambda(i)-i+\omega(i)$. Note that $\lambda^\omega$ is not necessarily a partition.

\begin{lemma} \label{L:k2c} Assume that $\ell(\lambda)\leq m$, then
$$g_{\mu,\nu}^\lambda = \sum_{\omega\in \mf{S}_m} \op{sgn}(\omega) \sum_{|\eta_i|=\lambda^\omega(i)} c_{\bs{\eta}}^\mu c_{\bs{\eta}}^\nu.$$
\end{lemma}

\begin{proof} We first recall that the Jacobi-Trudi formula \cite[A.5]{FH} expresses
$s_\lambda$ as the determinant of the matrix $\left( s_{(\lambda(i)-i+j)} \right)_{i,j}$. So
\begin{align*}s_\mu * s_\lambda &= s_\mu * \det(s_{(\lambda(i)-i+j)}) = s_\mu * \sum_{\omega\in \mf{S}_m} \op{sgn}(\omega)\prod_{i=1}^{m} s_{(\lambda(i)-i+\omega(i))} \\
&=\sum_{\omega\in \mf{S}_m} \op{sgn}(\omega)\sum_{|\eta_i|=\lambda^\omega(i)} c_{\bs{\eta}}^\mu c_{\bs{\eta}}^\nu s_\nu \quad \text{(inductively apply Littlewood's identity)}
\end{align*}
Therefore $\displaystyle g_{\mu,\nu}^\lambda = \sum_{\omega\in \mf{S}_m} \op{sgn}(\omega) \sum_{|\eta_i|=\lambda^\omega(i)} c_{\bs{\eta}}^\mu c_{\bs{\eta}}^\nu.$
\end{proof}

\begin{corollary}[{\cite[Lemma 3.1]{PP}}] \label{C:k2c} Let $(\lambda,\mu,\nu)$ be a triple of partition of $n$ with $\ell(\lambda)\leq 2$.
We set $a_k(\mu,\nu)=\sum_{|\eta_1|=k, |\eta_2|=n-k} c_{\eta_1,\eta_2}^\mu c_{\eta_1,\eta_2}^\nu$, then
$$g_{\mu,\nu}^\lambda = a_{\lambda(1)}(\mu,\nu)-a_{\lambda(1)+1}(\mu,\nu).$$
\end{corollary}

In practice, if we apply the above formula to compute Kronecker coefficients, we may find that most of the LR coefficients are zeros. So to really make this formula effective, we need to understand the non-vanishing conditions for LR coefficients.
Such conditions are conjectured by A. Horn, and solved by Klyachko and Knutson-Tao \cite{KT}, and many other people, remarkably \cite{DW2}.
To state these conditions, we define for any positive decreasing sequence $I=\{i_1,i_2,\cdots,i_r\}$ a partition $\lambda(I)=(i_1-r,i_2-(r-1),\dots,i_r-1).$

\begin{lemma}[{\cite[Theorem 1.5]{DW2}}] \label{L:Horn}
Let $(\lambda,\mu,\nu)$ be a triple of partitions satisfying $\max(\ell(\lambda,\ell(\mu),\ell(\nu))\leq m$ and $|\lambda|=|\mu|+|\nu|$.
Then
$$c_{\mu,\nu}^\lambda >0 \quad\Longleftrightarrow\quad
\sum_{i\in I} \lambda(i) \leq \sum_{j\in J} \mu(j) + \sum_{k\in K} \nu(k)$$
for any $I,J,K \subset \{1,2,\dots,m\}$ satisfying $|I|=|J|=|K|<m$ and
$c_{\lambda(J),\lambda(K)}^{\lambda(I)} =1$.
\end{lemma}
\noindent Since $c_{(j),(k)}^{(j+k)}=1$, as a rather trivial consequence of Lemma \ref{L:Horn},
we have that $\lambda(j+k-1)\leq \mu(j)+\nu(k)$ are all necessary. These inequalities are essentially due to H. Weyl.
It follows that $\ell(\mu)+\ell(\nu)\geq \ell(\lambda)$.
Another easy consequence on the length is that
$\ell(\lambda) \geq \max(\ell(\mu),\ell(\nu))$.
As an illustration, let us compute one family of Kronecker coefficients.


\begin{lemma} \label{L:g=2} We have that $g_{\mu,\nu}^{\lambda} = 2$ for
$\lambda=(3j,3k),\mu=(i+1,i,i-1)$, and $\nu=(j+1,j,j-1,k+1,k,k-1) $ with $j\geq k+2,k\geq 1$ and $i=j+k$.
\end{lemma}

\begin{proof} We assume that $k\geq 2$. The case of $k=1$ can be treated similarly.
By the restriction on the length, to make $c_{\eta_1,\eta_2}^\mu c_{\eta_1,\eta_2}^\nu$ contribute to $a_{*}(\mu,\nu)$ in Corollary \ref{C:k2c},
we must have that $\ell(\eta_1)=\ell(\eta_2)=3$.
Applying Weyl's inequalities for $c_{\eta_1,\eta_2}^\nu\neq 0$, we have that $\eta_2(r) \geq \nu(3+r)=k+2-r$ for $r=1,2,3$.
In particular, $|\eta_2|\geq 3k$.
But $|\eta_2|=\lambda(2)-\ep=3k-\ep$ in $a_{\lambda(1)+\ep}(\mu,\nu)$ for $\ep=0,1$.
So the only choice for $\eta_2$ is $(k+1,k,k-1)$ and $a_{\lambda(1)+1}(\mu,\nu)=0$.
Similarly applying Weyl's inequalities for $c_{\eta_1,\eta_2}^\mu\neq 0$,
we get $\eta_1(r)\geq j+1-r$.
We claim that the only choice for $\eta_1$ is $(j+1,j,j-1)$.
Since $|\eta_1|=3j$, we do not have many choices left.
Other choices can be easily ruled out by other Horn's inequalities.
We note that there are six more inequalities coming from
$c_{(0,0,(0,0)}^{(0,0)} = c_{(0,0),(1,0)}^{(1,0)} =c_{(0,0),(1,1)}^{(1,1)}=c_{(1,0),(1,0)}^{(1,1)} =1$ and their symmetry.
Finally, we compute by the Littlewood-Richardson rule that $c_{\eta_1,\eta_2}^\mu=2$~and~$c_{\eta_1,\eta_2}^\nu=1$.
Hence, $g_{\mu,\nu}^{\lambda} = 2$.

\end{proof}

\section{Semi-invariants of Quiver Representations} \label{S:SI}

\subsection{Schofield's Construction}
Let us briefly recall the semi-invariant rings of quiver representations \cite{S1}.
Let $Q$ be a finite quiver without oriented cycles.
For a dimension vector $\beta$ of $Q$, let $V$ be a $\beta$-dimensional vector space $\prod_{i\in Q_0} k^{\beta(i)}$. We write $V_i$ for the $i$-th component of $V$.
The space of all $\beta$-dimensional representations is
$$\Rep_\beta(Q):=\bigoplus_{a\in Q_1}\Hom(V_{t(a)},V_{h(a)}).$$
The product of general linear group
$$\GL_\beta:=\prod_{i\in Q_0}\GL(V_i)$$
acts on $\Rep_\beta(Q)$ by the natural base change.
Define $\SL_\beta\subset \GL_\beta$ by
$$\SL_\beta=\prod_{i\in Q_0}\SL(V_i).$$
We are interested in the rings of semi-invariants
$$\SI_\beta(Q):=k[\Rep_\beta(Q)]^{\SL_\beta}.$$
The ring $\SI_\beta(Q)$ has a weight space decomposition
$$\SI_\beta(Q)=\bigoplus_\sigma \SI_\beta(Q)_\sigma,$$
where $\sigma$ runs through the multiplicative {\em characters} of $\GL_\beta$.
We refer to such a decomposition the $\sigma$-grading of $\SI_\beta(Q)$.
Recall that any character $\sigma: \GL_\beta\to k^*$ can be identified with a weight vector
$\sigma \in \mb{Z}^{Q_0}$
\begin{equation} \label{eq:char} \big(g(i)\big)_{i\in Q_0}\mapsto\prod_{i\in Q_0} \big(\det g(i)\big)^{\sigma(i)}.
\end{equation}
Since $Q$ has no oriented cycles, the degree zero component is the field $k$ \cite{Ki}.

Let us understand these multihomogeneous components
$$\SI_\beta(Q)_\sigma:=\{f\in k[\Rep_\beta(Q)]\mid g(f)=\sigma(g)f, \forall g\in\GL_\beta \}.$$
For any projective presentation $f: P_1\to P_0$, we view it as an element in the homotopy category $K^b(\proj Q)$ of bounded complexes of projective representations of $Q$.
The {\em weight vector} $\f$ of $f$ is the corresponding element in the Grothendieck group of $K^b(\proj Q)$.
Concretely, suppose that $P_1=P(\f_1)$ and $P_0=P(\f_0)$ for $\f_1,\f_0\in\mathbb{N}_0^{Q_0}$, then $\f = \f_1-\f_0$.
Here, we use the notation $P(\f)$ for $\bigoplus_{i\in Q_0} \f(i) P_i^{}$, where $P_i$ is the indecomposable projective representation corresponding to the vertex $i$.
From now on, we will view a weight $\sigma$ as an element in the dual $\Hom_{\mb{Z}}(\mb{Z}^{Q_0},\mb{Z})$ via the usual dot product.

We assume that $\f=\sigma$ and $\sigma(\beta)=0$.
We apply the functor $\Hom_Q(-,N)$ to $f$ for $N\in\Rep_\beta(Q)$
\begin{equation} \label{eq:canseq} \Hom_Q(P_0,N)\xrightarrow{\Hom_Q(f,N)}\Hom_Q(P_1,N).
\end{equation}
Since $\sigma(\beta)=0$, $\Hom_Q(f,N)$ is a square matrix.
Following Schofield \cite{S1}, we define
$$s(f,N):=\det \Hom_Q(f,N).$$
We give a more concrete description for the map $\Hom_Q(f,N)$.\\
{\bf Concrete description of $\Hom_Q(f,N):\ $}
Recall that a morphism $P_1\xrightarrow{f} P_0$ can be represented by a matrix whose entries are linear combination of paths. Applying $\Hom_Q(-,N)$ to this morphism is equivalent to that we first transpose the matrix of $f$, then substitute paths in the matrix by corresponding matrix representations in $N$.

We set $s(f)(-)=s(f,-)$ as a function on $\Rep_\beta(Q)$. It is proved in \cite{S1} that $s(f)\in\SI_\beta(Q)_{\sigma}$.
In fact,
\begin{theorem}[\cite{DW1,SV,DZ}] \label{T:inv_span} $s(f)$'s span $\SI_\beta(Q)_{\sigma}$ over the base field $k$.
\end{theorem}

It is easy to see that if $s(f)\neq 0$, then $f$ resolves some representation $M$, say of dimension $\alpha$
$$0\to P_1 \xrightarrow{f} P_0\to M\to 0.$$
From the long exact sequence
\begin{equation} \label{eq:canseq} \Hom_Q(M,N)\hookrightarrow\Hom_Q(P_0,N)\xrightarrow{\Hom_Q(f,N)}\Hom_Q(P_1,N)\twoheadrightarrow\Ext_Q(M,N),
\end{equation}
we see that $\alpha$ and $\sigma$ are related by
$\sigma(-)=-\innerprod{\alpha,-}_Q$, where $\innerprod{-,-}_Q$ is the {\em Euler form} of $Q$.
In this case, we call $\sigma$ the weight vector corresponding to $\alpha$, and denote it by $ ^\sigma \alpha$;
and conversely we call $\alpha$ the dimension vector corresponding to $\sigma$, and denote it by $^\alpha \sigma$.
It also follows from \eqref{eq:canseq} that $s(f,N)\neq 0$ if and only if $\Hom_Q(M,N)=0$ or $\Ext_Q(M,N)=0$.

We want to point out that the function $s(f)$ is determined, up to a scalar multiple, by the homotopy equivalent class of $f$, and thus by the isomorphism class of $M$.
If one hopes to define the function $s(f)$ uniquely in terms of representations, one can take the {\em canonical resolution} $f_{M}$ of $M$.
This is Schofield's original definition in \cite{S1}.

\subsection{Stability} We define the subgroup $\GL_\beta^\sigma$ to be the kernel of the character map \eqref{eq:char}. The invariant ring of its action is $$\SI_\beta^\sigma(Q):=k[\Rep_\beta(Q)]^{\GL_\beta^\sigma}=\bigoplus_{n\geqslant 0} \SI_\beta(Q)_{n\sigma}.$$

\begin{definition}
A representation $M\in\Rep_\beta(Q)$ is called {\em $\sigma$-semi-stable}
if there is some non-constant $f\in \SI_\beta^\sigma(Q)$ such that $f(M)\neq 0$.
It is called {\em stable} if in addition the orbit $\GL_\beta^\sigma\cdot M$ is closed of dimension equal to $\dim\GL_\beta^\sigma-1$.
\end{definition}

\noindent Based on Hilbert-Mumford criterion, King provided a simple criterion for the stability of a representation.
\begin{lemma} \cite[Proposition 3.1]{Ki}  \label{L:King} A representation $M$ is $\sigma$-semi-stable (resp. $\sigma$-stable) if and only if $\sigma(\dimbar M)=0$ and $\sigma(\dimbar L)\geqslant 0$ (resp. $\sigma(\dimbar L)>0$) for any non-trivial subrepresentation $L$ of $M$.
\end{lemma}

Let $\Sigma_\beta(Q)$ be the set of all weights $\sigma$ such that $\SI_\beta(Q)_\sigma$ is non-empty.
The next theorem is an easy consequence of King's stability criterion and Theorem \ref{T:inv_span}.
By a general $\beta$-dimensional representation, we mean in a sufficiently small Zariski open subset (``sufficient" here depends on the context).
Following \cite{S2}, we use the notation $\gamma\hookrightarrow\beta$ to mean that
a general $\beta$-dimensional representation has a $\gamma$-dimensional subrepresentation.

\begin{theorem} \cite[Theorem 3]{DW1} \label{T:cone} We have
$$\Sigma_\beta(Q)=\{\sigma\in \Hom_{\mb{Z}}(\mb{Z}^{Q_0},\mb{Z}) \mid \sigma(\beta)=0 \text{ and } \sigma(\gamma)\geq 0 \text{ for all } \gamma\hookrightarrow \beta\}.$$
In particular, $\Sigma_\beta(Q)$ is a saturated semigroup.
\end{theorem}

%

Recall that a dimension vector $\alpha$ is called a {\em real} root if $\innerprod{\alpha,\alpha}_Q=1$.
It is called {\em Schur} if $\Hom_Q(M,M)=k$ for
$M$ general in $\Rep_\alpha(Q)$. For a real Schur root $\alpha$, $\Rep_{n\alpha}(Q)$ has a dense orbit for any $n\in \mb{N}$.

\begin{lemma}[{\cite[Lemma 1.7]{Fs1}}] \label{L:gsreal}
If $\alpha$ is a real Schur root, then $\SI_\beta(Q)_{-\innerprod{n\alpha,-}}\cong k$ for any $n\in \mb{N}$.
\end{lemma}

A weight is called {\em extremal} in $\Sigma_\beta(Q)$ if it lies on an extremal ray of $\mb{R}_+\Sigma_\beta(Q)$.

\begin{lemma}[{\cite[Lemma 1.8]{Fs1}}] \label{L:irreducible} If $\sigma$ is an indivisible extremal weight, then any semi-invariant function $s$ of weight $\sigma$ is irreducible.
\end{lemma}

%

\subsection{The Outer Action}
Let $A_{i,j}$ be the vector space spanned by arrows from $i$ to $j$.
Let $\Aut(Q_1)$ be the group $\prod_{i,j \in Q_0} \GL(A_{i,j})$
acting naturally on $\bigoplus_{i,j} A_{i,j}$.
This induces an action on the space of paths $kQ$, and thus induces an action on the space of presentations $\Hom_Q(P_1,P_0)$.
The group $\Aut(Q_1)$ also acts on the representation spaces $\Rep_\beta(Q)$ by twisting the action of arrows:
$$ma=m(ah) \quad\text{for $a\in Q_1$ and $m\in M$.}$$
We will write this action exponentially, that is, $M^h$ and $f^h$ for $M\in \Rep_\beta(Q)$ and $f\in \Hom_Q(P_1,P_0)$.
The action on the representations induces an action on the semi-invariant rings: $(h\cdot s)(M)=s(M^h)$,
because the action of $\Aut(Q_1)$ commutes with the action of $\GL_\beta$
$$(h\cdot s)(gM) = s\left((gM)^h\right) = s(gM^h) = \sigma(g) s(M^h) = \sigma(g) (h\cdot s)(M).$$
On the other hand, the action on the presentations also induces an action on the semi-invariant rings: $h\cdot s(f) = s(f^h)$. It turns out that two induced action are the same.
Indeed, we recall that $s(f)(N)=s(f,N)=\det\Hom_Q(f,N)$.
It is clear from our description of $\Hom_Q(f,N)$ that $s(f^{h^{-1}},N^h)=s(f,N)$,
so $(h\cdot s(f))(N) = s(f)(N^h) = s(f^h)(N)$.

\begin{definition}[\cite{Fg}] \label{D:Ginv} Let $G$ be a subgroup of $\Aut(Q_1)$. A presentation $f\in \Hom_Q(P_1,P_0)$ is called {\em coherently $G$-invariant} if there is an algebraic group homomorphism $\varphi: G\to \Aut_Q(P_1)\times \Aut_Q(P_0)$ such that $f^h = \varphi(h)\cdot f$.
\end{definition}

\begin{lemma}[{\em cf.} {\cite[4.1]{Fg}}] \label{L:Ginv}
If $f$ is coherently $G$-invariant, then $s(f)$ is~$G$-semi-invariant.
\end{lemma}

\begin{proof} It is proved in \cite[Proposition 5.13]{IOTW} that $s(-,N)$ is semi-invariant as a function on the presentation space $\Hom_Q(P_1,P_0)$ under the action of $\Aut_Q(P_1)\times \Aut_Q(P_0)$.
Suppose that $f^h = \varphi(h)\cdot f$, then
$s(f^h)(N)= s(\varphi(h)\cdot f)(N) = \sigma(\varphi(h)) s(f)(N)$ for some character $\sigma:\Aut_Q(P_1)\times \Aut_Q(P_0)\to k^*$.
Hence $h\cdot s(f)=(\sigma\varphi)(h) s(f)$, which means that $s(f)$ is $G$-semi-invariant.
\end{proof}

\section{Semi-invariants of Flagged Kronecker Quivers} \label{S:FKm}
\subsection{Relating Kronecker Coefficients}
Consider the following quiver $K_{l_1,l_2}^m$
$$\kronmll{b_{-1}}{}{}{b_1}{a_1,a_2,\dots,a_m}$$
We see that the full subquiver of two vertices $-l_1$ and $l_2$ is an $m$-arrow Kronecker quiver, so we call it a {\em flagged Kronecker quiver}. The full subquiver of negative (resp. positive) vertices are called its negative (positive) arm.

We fix a Schur root $\beta$ of $K_{l_1,l_2}^m$. It is easy to see that $\beta$ is nondecreasing (resp. nonincreasing) along the negative (positive) arm.
Let $V$ be a $\beta$-dimensional vector space, and $W:=A_{-l_1,l_2}$ be the $m$-dimensional vector space spanned by arrows from $-l_1$ to $l_2$.
Let us decompose the coordinate ring
\begin{align*}
& k\left[\Rep_\beta(K_{l_1,l_2}^m)\right] \\
= & k\Big[\bigoplus_{i={1}}^{l_1-1} \Hom(V_{-i},V_{-(i+1)}) \oplus \bigoplus_{i={1}}^{l_2-1} \Hom(V_{i+1},V_{i}) \oplus \Hom(V_{-l_1},V_{l_2})\otimes W \Big]\\
= & \bigotimes_{i={1}}^{l_1-1} \op{Sym}(V_{-i}\otimes V_{-(i+1)}^*) \otimes \bigotimes_{i={1}}^{l_2-1} \op{Sym}(V_{i+1}\otimes V_{i}^*) \otimes \op{Sym}(V_{-l_1}\otimes V_{l_2}^* \otimes W)\\
\intertext{By Cauchy's formula \cite{FH}, this decomposes into a direct sum over all tuples of partitions $\bs{\lambda}=(\lambda_{-l_1}, \dots, \lambda_{-1}, \lambda_0,\lambda_1,\dots,\lambda_{l_2})$ of the summands}
 & \bigotimes_{i={1}}^{l_1-1} \left( \S_{\lambda_{-i}}V_{-i}\otimes \S_{\lambda_{-i}}V_{-(i+1)}^* \right)\otimes \bigotimes_{i={1}}^{l_2-1} \left( \S_{\lambda_{i}}V_{i+1}\otimes \S_{\lambda_{i}} V_{i}^*\right) \otimes \S_{\lambda_0}(V_{-l_1}\otimes V_{l_2}^*) \otimes \S_{\lambda_0}W \\
\intertext{which is by Lemma \ref{L:SW}}
& \bigoplus_{|\lambda|=|\mu|=|\nu|=|\lambda_0|} g_{\mu,\nu}^\lambda \Bigg( \bigotimes_{i={1}}^{l_1-1} \left(\S_{\lambda_{-i}}V_{-i}\otimes \S_{\lambda_{-i}}V_{-(i+1)}^* \right) \otimes \bigotimes_{i={1}}^{l_2-1} \left( \S_{\lambda_{i}}V_{i+1}\otimes \S_{\lambda_{i}} V_{i}^* \right) \\
& \hspace{3.05in} \otimes \big( \S_{\mu}V_{-{l_1}}\otimes \S_{\nu}V_{l_2}^* \otimes \S_{\lambda}W\big) \Big).
\end{align*}
So the invariant ring $k[\Rep_\beta(K_{l_1,l_2}^m)]^{\SL_\beta}$
\begin{align*}
= & \bigoplus_{\substack{\bs{\lambda} \\ |\lambda|=|\mu|=|\nu|=|\lambda_0|} } g_{\mu,\nu}^\lambda (\S_{\lambda_{-1}}V_{-1})^{\SL(V_{-1})} \otimes (\S_{\lambda_{1}}V_{1}^*)^{\SL(V_{1})} \\
& \quad \otimes \bigotimes_{i={2}}^{l_1-1} (\S_{\lambda_{-i}}V_{-i}\otimes \S_{\lambda_{-(i-1)}}V_{-i}^*)^{\SL(V_{-i})}
\otimes \bigotimes_{i={2}}^{l_2-1} (\S_{\lambda_{i}}V_i\otimes \S_{\lambda_{i-1}} V_{i}^*)^{\SL(V_{i})} \notag \\
& \quad \otimes (\S_{\mu}V_{-(l_1-1)}\otimes \S_{\lambda_{-(l_1-1)}}V_{-l_1}^*)^{\SL(V_{-l_1})}
\otimes (\S_{\lambda_{1}}V_{l_2-1}\otimes \S_{\nu} V_{l_2}^*)^{\SL(V_{l_2})}\otimes \S_{\lambda}W. \notag
\end{align*}

\begin{lemma} \label{L:SI(Kml)} Let $\sigma$ be a weight on $K_{l_1,l_2}^m$,
and $\mu(\sigma),\nu(\sigma)$ be two partitions of $n$ equal to $\left(\beta(-l_1)^{-\sigma(-l_1)},\dots,\beta(-1)^{-\sigma(-1)}\right)^*$, and $\left(\beta(l_2)^{\sigma(l_2)}, \dots, \beta(1)^{\sigma(1)}\right)^*$, then
$$\SI_\beta(K_{l_1,l_2}^m)_\sigma = \bigoplus_{|\lambda|=n} g_{\mu(\sigma),\nu(\sigma)}^\lambda \S_{\lambda}W.$$
\end{lemma}

\begin{proof} The proof is similar to that of \cite[Proposition 1]{DW1}.
For $(\S_{\lambda_{-1}}V_{-1})^{\SL(V_{-1})}$ to contribute to $\SI_\beta(K_{l_1,l_2}^m)_\sigma$, the partition $\lambda_{-1}^*$ must be the square partition $\beta(-1)^{-\sigma(-1)}$.
In this case, $(\S_{\lambda_{-1}}V_{-1})^{\SL(V_{-1})}=k$.
For the next one $(\S_{\lambda_{-2}}V_{-2}\otimes \S_{\lambda_{-1}}V_{-2}^*)^{\SL(V_{-2})}$ to contribute,
the partition $\lambda_{-2}^*$ must be equal to $\lambda_{-1}^*$ with some extra rows of size $\beta(-2)^{-\sigma(-2)}$.
In this case, $(\S_{\lambda_{-2}}V_{-2}\otimes \S_{\lambda_{-1}}V_{-2}^*)^{\SL(V_{-2})}$ is also $k$.
We continue this procedure, and conclude that $\mu^*$ must be equal to $\left(\beta(-l_1)^{-\sigma(-l_1)},\dots,\beta(-1)^{-\sigma(-1)}\right)$.
Similarly, we have that $\nu^*=\left(\beta(l_2)^{\sigma(l_2)}, \dots, \beta(1)^{\sigma(1)}\right)$.
We thus get the required equality.
\end{proof}

Let $U$ and $T$ be the subgroups of upper-triangular and diagonal matrices in $\GL(W)$.
The semi-invariant ring $\SI_\beta(K_{l_1,l_2}^m)$ is also graded by the $T$-weights $\lambda$.
We write $\SI_\beta(K_{l_1,l_2}^m)_{\sigma,\lambda}$ for the weight $\lambda$ component of $\SI_\beta(K_{l_1,l_2}^m)_\sigma$.
We refer to this grading as the extended $\sigma$-grading or $\wtd\sigma$-grading.

\begin{corollary} \label{C:dimUinv} The dimension of $\SI_\beta(K_{l_1,l_2}^m)_{\sigma,\lambda}^U$ is equal to $g_{\mu(\sigma),\nu(\sigma)}^\lambda$.
\end{corollary}

We can also decompose the coordinate ring of $\SI_\beta(K_{l_1,l_2}^m)_\sigma$ in another way.
\begin{align*}
& k\left[\Rep_\beta(K_{l_1,l_2}^m)\right] \\
= & \bigotimes_{i={1}}^{l_1-1} \op{Sym}(V_{-i}\otimes V_{-(i+1)}^*) \otimes \bigotimes_{i={1}}^{l_2-1} \op{Sym}(V_{i+1}\otimes V_{i}^*) \otimes \op{Sym}(V_{-l_1}\otimes V_{l_2}^*)^{\otimes m}.\\
\intertext{By Cauchy's formula \cite{FH}, this decomposes into a direct sum over all tuples of partitions $\bs{\lambda}=(\lambda_{-l_1}, \dots, \lambda_{-1},\lambda_1,\dots,\lambda_{l_2})$ of the summands}
 & \bigotimes_{i={1}}^{l_1-1} \left( \S_{\lambda_{-i}}V_{-i}\otimes \S_{\lambda_{-i}}V_{-(i+1)}^* \right)\otimes \bigotimes_{i={1}}^{l_2-1} \left( \S_{\lambda_{i}}V_{i+1}\otimes \S_{\lambda_{i}} V_{i}^*\right) \otimes \left( \bigoplus_{\lambda_0} \S_{\lambda_0}V_{-l_1}\otimes \S_{\lambda_0} V_{l_2}^* \right)^{\otimes m} \\
=& \bigotimes_{i={1}}^{l_1-1} \left( \S_{\lambda_{-i}}V_{-i}\otimes \S_{\lambda_{-i}}V_{-(i+1)}^* \right)\otimes \bigotimes_{i={1}}^{l_2-1} \left( \S_{\lambda_{i}}V_{i+1}\otimes \S_{\lambda_{i}} V_{i}^*\right) \\
 & \hspace{1.83in} \otimes \left(\bigoplus_{\mu,\nu, \bs{\eta}=\{\eta_1,\eta_2,\dots,\eta_m\}} c_{\bs{\eta}}^\mu c_{\bs{\eta}}^\nu S^{\mu}V_{-l_1}\otimes S^{\nu} V_{l_2}^*\right).
\end{align*}
So the invariant ring $k[\Rep_\beta(K_{l_1,l_2}^m)]^{\SL_\beta}$
\begin{align*}
= & \bigoplus_{\substack{\bs{\lambda} \\ |\mu|=|\nu|=\sum_i|\eta_i| }} c_{\bs{\eta}}^\mu c_{\bs{\eta}}^\nu  (\S_{\lambda_{-1}}V_{-1})^{\SL(V_{-1})} \otimes (\S_{\lambda_{1}}V_{1}^*)^{\SL(V_{1})} \\
& \quad\! \otimes \bigotimes_{i={2}}^{l_1-1} (\S_{\lambda_{-i}}V_{-i}\otimes \S_{\lambda_{-(i-1)}}V_{-i}^*)^{\SL(V_{-i})}
\otimes \bigotimes_{i={2}}^{l_2-1} (\S_{\lambda_{i}}V_i\otimes \S_{\lambda_{i-1}} V_{i}^*)^{\SL(V_{i})} \notag \\
& \quad\! \otimes (\S_{\mu}V_{-(l_1-1)}\otimes \S_{\lambda_{-(l_1-1)}}V_{-l_1}^*)^{\SL(V_{-l_1})}
\otimes (\S_{\lambda_{1}}V_{l_2-1}\otimes \S_{\nu} V_{l_2}^*)^{\SL(V_{l_2})}. \notag
\end{align*}
\noindent By a similar argument as in Lemma \ref{L:SI(Kml)}, we obtain the following lemma.

\begin{lemma}  \label{L:SI(Kml)2} In the same setting as Lemma \ref{L:SI(Kml)}, we have that
$$\dim\SI_\beta(K_{l_1,l_2}^m)_\sigma = \sum_{\sum_{i=1}^m|\eta_i|=n} c_{\bs{\eta}}^\mu c_{\bs{\eta}}^\nu.$$
Moreover, each $\sum_{|\eta_i|=\lambda(i)} c_{\bs{\eta}}^\mu c_{\bs{\eta}}^\nu$ counts the dimension of $\SI_\beta(K_{l_1,l_2}^m)_{\sigma,\lambda}$.
\end{lemma}

\noindent It follows from Lemma \ref{L:SI(Kml)2} and \ref{L:k2c} that
\begin{corollary} \label{C:SI(Kml)2} We have that
$$g_{\mu(\sigma),\nu(\sigma)}^\lambda = \sum_{\omega\in \mf{S}_m} \op{sgn}(\omega) \dim \SI_\beta(K_{l_1,l_2}^m)_{\sigma,\lambda^\omega}.$$
\end{corollary}

\subsection{$m=2$} For the rest of the paper, we will mainly concern about the case when $l_1=l_2=l,m=2,$ and $\beta$ is the {\em standard dimension vector}
$\beta_l(i)=|i|$.
We note that for any pair of partitions $\mu,\nu$ with $\ell(\mu),\ell(\nu)\leq l$,
there is a weight $\sigma$ such that $\mu=\mu(\sigma),\nu=\nu(\sigma)$ in Lemma \ref{L:SI(Kml)}.
Let $\e_i \in \mb{Z}^{2l}$ be the unit vector supported on the vertex $i$. We use the convention that $\e_0$ is the zero vector.
It follows from Theorem \ref{T:cone} that if $\sigma\in \Sigma_\bl(K_{l,l}^2)$, then
\begin{equation} \label{eq:signsigma}
\text{$\sigma(i)$ must be nonnegative (resp. nonpositive) if $i$ is positive (negative).}
\end{equation}
We consider the following set of weight vectors labeled by $(i;j,k)$ and $(j,k;i)$ satisfying $i=j+k$:
\begin{align*}
\tag{I} &\f_{i;j,k}:=\e_j+\e_k-\e_{-i}; \\
\tag{II} &\f_{j,k;i}:=\e_i-\e_{-j}-\e_{-k}.
\end{align*}

\noindent For each weight vector $\f_{i;j,k}$, we associate a projective presentation $f_{i;j,k}$ given by
\begin{align*}
 P_{j}\oplus P_{k}  & \xrightarrow{\sm{p_{-i,j}^1 \\p_{-i,k}^{2}} }  P_{-i}, \\
\intertext{where $p_{-i,j}^\ep$ is the unique path from $-i$ to $j$ via the arrow $a_\ep$.
We also use the convention that $P_0=0$ in accordance with $\e_0=0$. Similarly, for each $\f_{j,k;i}$ we associate a presentation}
 P_{i} & \xrightarrow{\sm{p_{-j,i}^1 & p_{-k,i}^{2}} }  P_{-j}\oplus P_{-k}.
\end{align*}

\noindent To simplify our exposition, we denote $-(i;j,k):=(j,k;i)$ and write $f_{i;j,k}^\pm$ for both $f_{i;j,k}$ and $f_{j,k;i}$.
We introduce an involution $-$ on $kQ$ sending the path $p_{-i,j}^\ep$ to $p_{-j,i}^\ep$.
This induces a map
\begin{equation} \label{eq:-map}
-: \Hom_Q(\bigoplus_j P_j, \bigoplus_i P_{-i}) \xrightarrow{} \Hom_Q(\bigoplus_i P_{i}, \bigoplus_j P_{-j}).
\end{equation}
In this language, we can write $f_{i;j,k}^-$ as $-f_{i;j,k}$.

Recall that $s(f)$ is the Schofield's semi-invariant function attached to a presentation $f$.
We write $s_{i;j,k}^\pm$ for $s(f_{i;j,k}^\pm)$.
It is useful to study the $\GL(W)$-action on these special semi-invariants.
Recall that the $\GL(W)\cong \GL_2$ acts on $\Rep(K_{l_1,l_2}^m)$ by
$$\sm{t_1 & u\\ u'& t_2}: (A_1,A_2,B_i) \mapsto (t_1A_1+u'A_2, t_2A_2+uA_1,B_i),$$
where $A_1,A_2$ are the matrices of $a_1,a_2$ and $B_i$ are matrices on the arms.
Let $\omega$ be the transposition in the Weyl group $\mf{S}_2 \subset \GL_2$.
The following lemma is obvious.
\begin{lemma}  \label{L:WT-var} The element $\omega$ permutes $s_{i;j,k}$ and $s_{i;k,j}$.
The $T$-action is given by
\begin{equation}
(t_1,t_2) s_{i;j,k}^{\pm} = t_1^j t_2^k s_{i;j,k}^{\pm}.
\end{equation}
\end{lemma}

\begin{lemma} \label{L:level1}
Each $\f:=\f_{i;j,k}^\pm$ is extremal in $\Sigma_\bl(K_{l,l}^2)$.
Moreover, for $j\geq k$, $\SI_\beta(K_{l,l}^2)_{\f}^{U}$ is spanned by $s(f_{i;j,k}^{\pm})$.
\end{lemma}
\begin{proof} By Lemma \ref{L:SI(Kml)}, $\f$ corresponds to the partitions $(\mu,\nu)=\left((i)^*,(j,k)^*\right)$.
Since $\S^{(i)}\otimes \S^{(j,k)} = \S^{(j,k)}$, $\f$ lies in $\Sigma_\bl(K_{l,l}^2)$ and $\dim \SI_\beta(K_{l,l}^2)_{\f}^{U}= 1$.
$\f$ must be extremal because of \eqref{eq:signsigma}.
We observe that $f_{i;j,k}^{\pm}$ is a coherently $U$-invariant presentation
with $\varphi\left(\sm{1 & u \\ 0 &1}\right) = \left(\sm{e_j & up_{j,k} \\ 0 & e_k}, e_{-i}\right)$ or $\left(e_i,\sm{e_{-j} & up_{-k,-j} \\ 0 & e_{-k}}\right)$ depending on the sign of $f_{i;j,k}^{\pm}$.
Here $e_i$ is the trivial path corresponding to the vertex $i$.
Hence $s(f_{i;j,k}^\pm)$ is $U$-invariant and spans $\SI_\beta(K_{l,l}^2)_{\f}^{U}$.
\end{proof}

\begin{remark} Let $U'$ be the subgroups of lower-triangular matrices in $\GL_2$.
It follows that for $j\leq k$, $s(f_{i;j,k}^{\pm})$ is $U'$-invariant, and $s_{2j;j,j}^{\pm}$ is thus $\GL_2$-semi-invariant.
The latter statement is also clear from the fact that $\f_{2j;j,j}^{\pm}$ corresponds to an exceptional representation \cite[1.4]{Fg}.
We will see in Section \ref{S:CS} that $\{s_{i;j,k}^{\pm}\}$ is our choice of an (initial) cluster.
\end{remark}

\begin{example}
Consider the quiver $K_{2,2}^2$
$$\krontwotwotwo{a_1}{a_2}{b_{-1}}{b_1}$$
We list below the concrete form for the semi-invariant functions $s_{i;j,k}^{\pm}$.
\begin{align*}
s_{1;1,0}&=\det(B_{-1}A_1B_1), & s_{1;0,1}&=\det(B_{-1}A_2B_1); \\
s_{2;1^2}&=\det(A_1B_1,A_2B_1), & s_{1^2;2}&=\det\sm{B_{-1}A_1 \\ B_{-1}A_2};\\
s_{2;2,0}&=\det(A_1), & s_{2;0,2}&=\det(A_2).
\end{align*}
\end{example}

\section{Graded Cluster Algebras}  \label{S:CA}

\subsection{Cluster Algebras}
We follow mostly Section 3 of \cite{FP}.
The combinatorial data defining a cluster algebra is encoded in an {\em ice} quiver $\Delta$ with no loops or oriented 2-cycles.
The first $p$ vertices of $\Delta$ are designated as {\em mutable}; the remaining $q-p$ vertices are called {\em frozen}.
If we require no arrows between frozen vertices, then
such a quiver is uniquely determined by its {\em $B$-matrix} $B(\Delta)$.
It is a $p\times q$ matrix given by
$$b_{u,v} = |\text{arrows }u\to v| - |\text{arrows }v \to u|.$$

\begin{definition} \label{D:Qmu}
Let $u$ be a mutable vertex of $\Delta$.
The {\em quiver mutation} $\mu_u$ transforms $\Delta$ into the new quiver $\Delta'=\mu_u(\Delta)$ via a sequence of three steps.
\begin{enumerate}
\item For each pair of arrows $v\to u\to w$, introduce a new arrow $v\to w$ (unless both $v$ and $w$ are frozen, in which case do nothing);
\item Reverse the direction of all arrows incident to $u$;
\item Remove all oriented 2-cycles.
\end{enumerate}
\end{definition}



\begin{definition} \label{D:seeds} 
Let $\mc{F}$ be a field containing $k$.
A {\em seed} in $\mc{F}$ is a pair $(\Delta,\b{x})$ consisting of an ice quiver $\Delta$ as above together with a collection $\b{x}=\{x_1,x_2,\dots,x_q\}$, called an {\em extended cluster}, consisting of algebraically independent (over $k$) elements of $\mc{F}$, one for each vertex of $\Delta$.
The elements of $\b{x}$ associated with the mutable vertices are called {\em cluster variables}; they form a {\em cluster}.
The elements associated with the frozen vertices are called
{\em frozen variables}, or {\em coefficient variables}.

A {\em seed mutation} $\mu_u$ at a (mutable) vertex $u$ transforms $(\Delta,\b{x})$ into the seed $(\Delta',\b{x}')=\mu_u(\Delta,\b{x})$ defined as follows.
The new quiver is $\Delta'=\mu_u(\Delta)$.
The new extended cluster is
$\b{x}'=\b{x}\cup\{x_{u}'\}\setminus\{x_u\}$
where the new cluster variable $x_u'$ replacing $x_u$ is determined by the {\em exchange relation}
\begin{equation} \label{eq:exrel}
x_u\,x_u' = \prod_{v\rightarrow u} x_v + \prod_{u\rightarrow w} x_w.
\end{equation}
\end{definition}

\noindent We note that the mutated seed $(\Delta',\b{x}')$ contains the same
coefficient variables as the original seed $(\Delta,\b{x})$.
It is easy to check that one can recover $(\Delta,\b{x})$
from $(\Delta',\b{x}')$ by performing a seed mutation again at $u$.
Two seeds $(\Delta,\b{x})$ and $(\Delta',\b{x}')$ that can be obtained from each other by a sequence of mutations are called {\em mutation-equivalent}, denoted by $(\Delta,\b{x})\sim (\Delta',\b{x}')$.

\begin{definition}[{\em Cluster algebra}] \label{D:cluster-algebra}
The {\em cluster algebra $\mc{C}(\Delta,\b{x})$} associated to a seed $(\Delta,\b{x})$ is defined as the subring of $\mc{F}$
generated by all elements of all extended clusters of the seeds mutation-equivalent to $(\Delta,\b{x})$.
\end{definition}


\noindent Note that the above construction of $\mc{C}(\Delta,\b{x})$ depends only, up to a natural isomorphism, on the mutation equivalence class of the initial quiver $\Delta$. So we may drop $\b{x}$ and simply write $\mc{C}(\Delta)$.


\subsection{Upper Cluster Algebras}

An amazing property of cluster algebras is
\begin{theorem}[{\em Laurent Phenomenon}, \textrm{\cite{FZ1,BFZ}}] \label{T:Laurent}
Any element of a cluster algebra $\mc{C}(\Delta,\b{x})$ can be expressed in terms of the
extended cluster $\b{x}$ as a Laurent polynomial, which is polynomial in coefficient variables.
\end{theorem}

Since $\mc{C}(\Delta,\b{x})$ is generated by cluster variables from the seeds mutation equivalent to $(\Delta,\b{x})$,
Theorem \ref{T:Laurent} can be rephrased as
$$\mc{C}(\Delta,\b{x}) \subseteq \bigcap_{(\Delta',\b{x}') \sim (\Delta,\b{x})}\mc{L}_{\b{x}'},$$
where $\mc{L}_{\b{x}}:=k[x_1^{\pm 1},\dots,x_p^{\pm 1}, x_{p+1}, \dots x_{q}]$.
Note that our definition of $\mc{L}_{\b{x}}$ is slightly different from the original one in \cite{BFZ}, where the coefficient variables are inverted in $\mc{L}_{\b{x}}:=k[x_1^{\pm 1},\dots,x_p^{\pm 1}, x_{p+1}^{\pm 1}, \dots x_{q}^{\pm 1}]$.
\begin{definition}[{\em Upper Cluster Algebra}]
The upper cluster algebra with seed $(\Delta,\b{x})$ is
$$\br{\mc{C}}(\Delta,\b{x}):=\bigcap_{(\Delta',\b{x}') \sim (\Delta,\b{x})}\mc{L}_{\b{x}'}.$$
\end{definition}

%
%
%

Any (upper) cluster algebra, being a subring of a field, is an integral
domain (and under our conventions, a $k$-algebra).
Conversely, given such a domain~$R$, one may be interested in
identifying $R$ as an (upper) cluster algebra.
As an ambient field~$\mc{F}$,
we can always use the quotient field~$\op{QF}(R)$.
The challenge is to find a seed $(\Delta,\b{x})$ in $\op{QF}(R)$ such
that $\br{\mc{C}}(\Delta,\b{x})=R$. The following lemma is very helpful.

\begin{lemma} \cite[Corollary 3.7]{FP} \label{L:RCA}
Let $R$ be a finitely generated unique factorization domain over $k$.
Let $(\Delta,\b{x})$ be a seed in the quotient field of $R$ such that all elements of $\b{x}$ and all elements of clusters adjacent to $\b{x}$ are irreducible elements of $R$. Then $R \supseteq \br{\mc{C}}(\Delta,\b{x})$.
\end{lemma}

\begin{remark} The original conclusion in \cite[Corollary 3.7]{FP} is that $R \supseteq \mc{C}(\Delta,\b{x})$. However, the proof implies this stronger result (see the comment before the proof of \cite[Proposition 3.6]{FP}).
\end{remark}

\subsection{Gradings}
\begin{definition} \label{D:wtconfig} A {\em weight configuration} $\bs{\sigma}$ on an ice quiver $\Delta$ is an assignment for each vertex $v$ of $\Delta$ a (weight) vector $\bs{\sigma}(v)$ such that for each mutable $u$, we have that
\begin{equation} \label{eq:weightconfig}
\sum_{v\to u} \bs{\sigma}(v) = \sum_{u\to w} \bs{\sigma}(w).
\end{equation}
A {\em mutation} $\mu_u$ at a mutable vertex $u$ transforms $\bs{\sigma}$ into a weight configuration $\bs{\sigma}'$ of the mutated quiver $\mu_u(\Delta)$ defined as
\begin{equation} \label{eq:mu_wt}
\bs{\sigma}'(v) = \begin{cases} \displaystyle \sum_{u\to w} \bs{\sigma}(w) - \bs{\sigma}(u) & \text{if } v=u, \\ \bs{\sigma}(v) & \text{if } v\neq u. \end{cases}
\end{equation} \end{definition}

\noindent By slight abuse of notation, we can view $\bs{\sigma}$ as a matrix whose $v$-th row is the weight vector $\bs{\sigma}(v)$.
In matrix notation, the condition \eqref{eq:weightconfig} is equivalent to $B\bs{\sigma}$ is a zero matrix.
It is not hard to see that for any weight configuration of $\Delta$, the mutation can be iterated.

\begin{definition} \label{D:full} A weight configuration $\bs{\sigma}$ is called {\em full} if the null rank of $B^{\T}$ is equal to the rank of $\bs{\sigma}$.
The null space of $B^{\T}$ is called the {\em grading space} of the cluster algebra $\mc{C}(\Delta)$.
\end{definition}

Given a weight configuration $\bs{\sigma}$ of $\Delta$,
we can assign a multidegree (or weight) to the (upper) cluster algebra $\mc{C}(\Delta,\b{x})$ by setting
$\deg(x_v)=\bs{\sigma}(v)$ for $v=1,2,\dots,q$.
Then mutation preserves multihomogeneousity.
We say that this (upper) cluster algebra is $\bs{\sigma}$-graded, and denoted by $\mc{C}(\Delta,\b{x};\bs{\sigma})$.
We refer to $(\Delta,\b{x};\bs{\sigma})$ as a graded seed.

\section{Cluster Characters from Quivers with Potentials} \label{S:CC}
\subsection{Quivers with Potentials}
In \cite{DWZ1} and \cite{DWZ2}, the mutation of quivers with potentials is invented to model the cluster algebras.
We emphasis here that ice quivers considered here can have arrows between frozen vertices.
Following \cite{DWZ1}, we define a potential $W$ on an ice quiver $\Delta$ as a (possibly infinite) linear combination of oriented cycles in $\Delta$.
More precisely, a {\em potential} is an element of the {\em trace space} $\Tr(\ckQ):=\ckQ/[\ckQ,\ckQ]$,
where $\ckQ$ is the completion of the path algebra $k\Delta$ and $[\ckQ,\ckQ]$ is the closure of the commutator subspace of $\ckQ$.
The pair $(\Delta,W)$ is an {\em ice quiver with potential}, or IQP for short.
For each arrow $a\in \Delta_1$, the {\em cyclic derivative} $\partial_a$ on $\widehat{k\Delta}$ is defined to be the linear extension of
$$\partial_a(a_1\cdots a_d)=\sum_{k=1}^{d}a^*(a_k)a_{k+1}\cdots a_da_1\cdots a_{k-1}.$$
For each potential $W$, its {\em Jacobian ideal} $\partial W$ is the (closed two-sided) ideal in $\ckQ$ generated by all $\partial_a W$.
The {\em Jacobian algebra} $J(\Delta,W)$ is the quotient algebra $\widehat{k\Delta}/\partial W$.
If $W$ is polynomial and $J(\Delta,W)$ is finite-dimensional, then the completion is unnecessary to define $J(\Delta,W)$.
This is the case throughout this paper.

The key notion introduced in \cite{DWZ1,DWZ2} is the {\em mutation} of quivers with potentials and their decorated representations.
For an ice quiver with {\em nondegenerate potential} (see \cite{DWZ1}), the mutation in certain sense ``lifts" the mutation in Definition \ref{D:Qmu}.
Since we do not need mutation in an explicit way, we refer readers to the original text.

\begin{definition} A {\em decorated representation} of a Jacobian algebra $J:=J(\Delta,W)$ is a pair $\mc{M}=(M,M^+)$,
where $M\in \Rep(J)$, and $M^+$ is a finite-dimensional $k^{\Delta_0}$-module.
\end{definition}

Let $\mc{R}ep(J)$ be the set of decorated representations of $J(\Delta,W)$ up to isomorphism. There is a bijection between two additive categories $\mc{R}ep(J)$ and $K^2(\proj J)$ mapping any representation $M$ to its minimal presentation in $\Rep(J)$, and the simple representation $S_u^+$ of $k^{\Delta_0}$ to $P_u\to 0$.
Suppose that $\mc{M}$ corresponds to a projective presentation
$P(\beta_1)\to P(\beta_0)$.

\begin{definition} The {\em $\g$-vector} $\g(\mc{M})$ of a decorated representation $\mc{M}$ is the weight vector of its image in $K^b(\proj J)$, that is, $\g=\beta_1-\beta_0$.
\end{definition}

\begin{definition}[\cite{DWZ1}] A potential $W$ is called {\em rigid} on a quiver $\Delta$ if
every potential on $\Delta$ is cyclically equivalent to an element in the Jacobian ideal $\partial W$.
Such a QP $(\Delta,W)$ is also called {\em rigid}.
\end{definition}

\noindent It is known \cite[Proposition 8.1, Corollary 6.11]{DWZ1} that every rigid QP is $2$-acyclic, and the rigidity is preserved under mutations. In particular, any rigid QP is nondegenerate.

\subsection{The Generic Cluster Character}
From now on, we assume that $(\Delta,W)$ is a nondegenerate IQP.
Let $\b{x}=\{x_1,x_2,\dots,x_q\}$ be an (extended) cluster.
For a vector $\g\in \mb{Z}^q$, we write $\b{x}^\g$ for the monomial $x_1^{\g(1)}x_2^{\g(2)}\cdots x_q^{\g(q)}$.
For $u=1,2,\dots,p$, we set ${y}_u= \b{x}^{-b_{u}}$ where $b_u$ is the $u$-th row of the matrix $B(\Delta)$,
and let ${\b{y}}=\{{y}_1,{y}_2,\dots,{y}_p\}$.

Suppose that an element $z\in\br{\mc{C}}(\Delta)$ can be written as
\begin{equation}\label{eq:z} z = \b{x}^{\g(z)} F({y}_1,{y}_2,\dots,{y}_p),
\end{equation}
where $F$ is a rational polynomial not divisible by any ${y}_i$, and $\g(z)\in \mb{Z}^q$.
If we assume that the matrix $B(\Delta)$ has full rank,
then the elements ${y}_1,{y}_2,\dots,{y}_p$ are algebraically independent so that the vector $\g(z)$ is uniquely determined \cite{FZ4}.
We call the vector~$\g(z)$~the (extended) {\em $\g$-vector} of $z$.
Definition implies at once that for two such elements $z_1,z_2$ we have that
$\g(z_1z_2) = \g(z_1) + \g(z_2)$.
So the set $G(\Delta)$ of all $\g$-vectors in $\br{\mc{C}}(\Delta)$ forms a sub-semigroup of $\mb{Z}^q$.

\begin{lemma}[{\cite[Lemma 5.5]{Fs1}, {\em cf.} \cite{P}}] \label{L:independent} Assume that the matrix $B(\Delta)$ has full rank.
Let $Z=\{z_1,z_2,\dots,z_k\}$ be a subset of $\br{\mc{C}}(\Delta)$ with well-defined $\g$-vectors.
If $\g(z_i)$'s are all distinct, then $Z$ is linearly independent over $k$.
\end{lemma}

\begin{definition}
To any $\g\in\mathbb{Z}^{\Delta_0}$ we associate the {\em reduced} presentation space $$\PHom_J(\g):=\Hom_J(P([\g]_+),P([-\g]_+)),$$
where $[\g]_+$ is the vector satisfying $[\g]_+(u) = \max(\g(u),0)$.
We denote by $\Coker(\g)$ the cokernel of a general presentation in $\PHom_J(\g)$.
\end{definition}
\noindent Reader should be aware that $\Coker(\g)$ is just a notation rather than a specific representation.
If we write $M=\Coker(\g)$, this simply means that we take a general presentation in $\PHom_J(\g)$, then let $M$ to be its cokernel.

\begin{definition} \label{D:mu_supported}
A representation is called {\em $\mu$-supported} if its supporting vertices are all mutable.
A weight vector $\g\in K_0(\proj J)$ is called {\em $\mu$-supported} if $\Coker(\g)$ is $\mu$-supported.
Let $G(\Delta,W)$ be the set of all $\mu$-supported vectors in $K_0(\proj J)$.
\end{definition}
\noindent It turns out that for a large class of IQP the set $G(\Delta,W)$ are given by lattice points in some rational polyhedral cone.
Such a class includes the ice hive quiver $\Delta_n$ introduced in \cite{Fs1} and the quiver $\Diamond_l$ to be introduced in the next section.

\begin{definition}[\cite{P}]
We define the {\em generic character} $C_W:G(\Delta,W)\to \mb{Z}(\b{x})$~by
\begin{equation} \label{eq:genCC}
C_W(\g)=\b{x}^{\g} \sum_{\e} \chi\big(\Gr^{\e}(\Coker(\g)) \big) {\b{y}}^{\e},
\end{equation}
where $\Gr^{\e}(M)$ is the variety parameterizing $\e$-dimensional quotient representations of $M$, and $\chi(-)$ denotes the topological Euler-characteristic.
\end{definition}
\noindent We note that if $(\Delta,\b{x})$ is $\bs{\sigma}$-graded, then
$C_W(\g)$ is multihomogeneous of degree $\g\bs{\sigma}$.

\begin{theorem}[{\cite[Corollary 5.14]{Fs1}, {\em cf.} \cite[Theorem 1.1]{P}}] \label{C:GCC} Suppose that IQP $(\Delta,W)$ is non-degenerate and $B(\Delta)$ has full rank.
The generic character $C_W$ maps $G(\Delta,W)$ (bijectively) to a set of linearly independent elements in $\br{\mc{C}}(\Delta)$ containing all cluster monomials.
\end{theorem}
%


\begin{definition} \label{D:model} We say that an IQP $(\Delta,W)$ {\em models} an algebra $\mc{A}$ if the generic cluster character maps $G(\Delta,W)$ (bijectively) onto a basis of $\mc{A}$.
If $\mc{A}$ is the upper cluster algebra $\uca(\Delta)$, then we simply say that $(\Delta,W)$ is a {\em cluster model}.
\end{definition}


\section{The Diamond Quivers} \label{S:Diamond}
\subsection{The Construction}

We define an ice quiver $\Diamond_l$ whose vertices are labeled by
$(i;j,k)$ and $(j,k;i)$ where $i,j,k\in \mb{Z}_{\geq 0}$ and $i=j+k$.
We also use the convention that $(i;i,0)=(i,0;i)$ and $(i;0,i)=(0,i;i)$,
and notation that $(2j; j,j):= (2j; j^2)$ and $(j,j;2j):= (j^2;2j)$.
Instead of $(\Diamond_l)_0$, we write $\Diamond_l^0$ for the set of vertices of $\Diamond_l$.

We define an involution on the set of vertices
\begin{align*}
& -: (i;j,k) \leftrightarrow (j,k;i).
\end{align*}
We will use the notation $\pm(i;j,k)$ for both $(i;j,k)$ and $-(i;j,k)=(j,k;i)$.
Note that there is another involution $\pm(i;j,k) \leftrightarrow \pm(i;k,j)$.
We freeze the vertices with $i=l$, i.e., of label $\pm(l;j,l-j)$.
A vertex labeled by $(i;i,0)$ or $(i;0,i)$ is called a {\em horizontal vertex};
a vertex labeled by $\pm (2j;j^2)$ is called a {\em vertical vertex}.
A vertex is called {\em special} if $|j-k|=1$.
Any rest vertex is called {\em general}.

We draw arrows according to the following recipe.
\begin{align*}
& \emph{Type $A$ arrows} & \pm(i;j,k) &\to \pm(i-1;j-1,k);\\
& \emph{Type $B$ arrows} & \pm(i;j,k) &\to \pm(i+1;j,k+1); \\
& \emph{Type $C$ arrows} & \pm(i;j,k) &\to \pm(i;j+1,k-1).
\end{align*}
However, we draw two (type $C$) arrows instead of one from $(1;0,1)$ to $(1;1,0)$.

When actually drawing the quiver, as a convention we put horizontal (resp. vertical) vertices horizontally as on the $x$-axis (vertically as on the $y$-axis);
put positive (resp. negative) vertices over (below) the $x$-axis. The sign of the $x$-coordinate is determined by the sign of $k-j$.
Unfortunately, we have to bend the type $C$ arrows between special vertices.
An arrow is called {\em nonnegative} (resp. {\em nonpositive}) if it lies on or above (resp. on or below) the $x$-axis.
As an illustration, $\Diamond_4$ is display in Figure \ref{F:Diamond4}. The type $A$ (resp. $B$ and $C$) arrows are in red (resp. blue and yellow).

\begin{figure}[h]
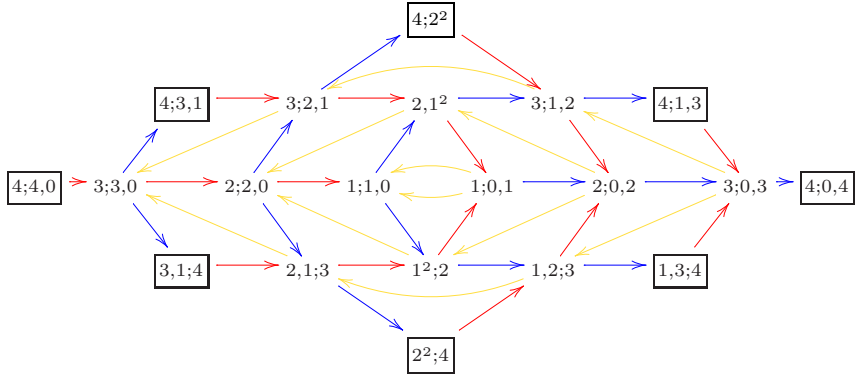

$$\ckrontwofourfour$$
\caption{The quiver $\Diamond_4$} \label{F:Diamond4}
\end{figure}
The next two lemmas are straightforward.

\begin{lemma} The $B$-matrix of $\Diamond_l$ is of full rank.
\end{lemma}

\noindent Recall the weight vectors $\f_{i;j,k}^{\pm}$ defined in Section \ref{S:FKm}.
\begin{lemma} \label{L:WC} The assignment $\pm(i;j,k)\mapsto \f_{i;j,k}^{\pm}$ on the vertices of $\Diamond_l$
defines a weight configuration $\bs{\sigma}_l$ on $\Diamond_l$.
Moreover, $\bs{\sigma}_l$ can be extended to a full weight configuration $\wtd{\bs{\sigma}}_l$
by assigning $(\f_{i;j,k}^{\pm},j,k)$ to $\pm(i;j,k)$.
\end{lemma}


Let $a$ (resp. $b$ and $c$) denote the sum of all type $A$ ($B$ and $C$) arrows.
Let $a_+$ and $a_-$ denote the sum of all nonnegative and nonpositive type $A$ arrows, and similar for $b_\pm$ and $c_\pm$.
We put the potential
$$W_l=(a_+b_+c_+-a_+c_+b_+)+(a_-b_-c_--a_-c_-b_-)$$
on the quiver $\Diamond_l$.
Then the Jacobian ideal is generated by the elements
\begin{align}
\label{eq:rel_Jn1} &e_u(ab-ba), e_u(bc-cb), e_u(ca-ac) & & \text{for $u$ mutable and non-horizontal,} \\
\label{eq:rel_Jn21} &e_u(ab - ba), e_u(b_+c_+-b_-c_-) & & \text{for $u=(j;j,0)$,} \\
\label{eq:rel_Jn22} &e_u(bc -cb), e_u(c_+a_+-c_-a_-) & & \text{for $u=(j;0,j)$,} \\
\label{eq:rel_Jn31} &e_u b_\pm a_\pm & & \text{for $u=(1;1,0)$,} \\
\label{eq:rel_Jn32} &e_u(b_\pm c_\pm-c_\pm b_\pm) & & \text{for $u=(1;0,1)$,} \\
\label{eq:rel_Jn4} & e_vac, cbe_v & & \text{for $v$ frozen.}
\end{align}

\noindent We say a path $p$ in $\Diamond_l$ {\em nonnegative} (resp. {\em nonpositive}) if all its composite arrows are nonnegative (nonpositive).
Such a path can be (uniquely) written as $e_{v_0} b_\pm^y c_\pm^z a_\pm^x$ by \eqref{eq:rel_Jn1}--\eqref{eq:rel_Jn4}.
Moreover, if a path is equivalent to another path in the Jacobian algebra $J_l:=J(\Diamond_l,W_l)$, then they must have the same length unless they are equivalent to the zero path.

\begin{lemma} The IQP $(\Diamond_l,W_l)$ is rigid.
\end{lemma}

\begin{proof} We need to show that every cycle up to cyclic equivalence is zero in the Jacobian algebra. By cyclic equivalence and relations \eqref{eq:rel_Jn1}--\eqref{eq:rel_Jn22}, cycles of form $e_u b_\pm c_\pm a_\pm$ can be written as $e_v a_\pm c_\pm b_\pm$ for some frozen $v$, which is zero in the $J_l$ by \eqref{eq:rel_Jn4}.

Now let $s$ be any cycle in $\Diamond_l$. We must have that the numbers of three types of arrows are the same in $s$. We can write $s$ as $s=p_1p_2\cdots p_n$, where each $p_i$ is either nonnegative or nonpositive.
So each $p_i$ can be written as $e_{v_i} b_\pm^y c_\pm^z a_\pm^x$.
We first assume that none of $x,y,z$ is zero.
We observe that $a,b,c$ can change order except for $bce_w$ and $e_wca$ with $w$ on the $x$-axis.
But these two cases cannot happen simultaneously in the procedure of moving $b$ right and moving $a$ left in $e_{v_i} b_\pm^y c_\pm^z a_\pm^x$.
So we can get at least one cycle of form $e_u b_\pm c_\pm a_\pm$.
We have treated the cases when $n=1$ and when some $p_i$ have all types of arrows.

Suppose that $n>1$. Since $s$ is a cycle, there is some $p_i$ such that $p_i=e_{v_0} p_i e_{v_t}$ with both $v_0$ and $v_t$ on the $x$-axis.
We can also assume that $p_i$ has only two types of arrows.
But $p_i$ cannot be of form $e_{v_0}b_\pm^yc_\pm^z e_{v_t}$ or $e_{v_0}c_\pm^za_\pm^x e_{v_t}$, because otherwise the arrow preceding or succeeding $p_i$ must be of the third type on the $x$-axis.
So let us assume that $p_i=e_{v_0}b_\pm^ya_\pm^xe_{v_t}$.
Then by \eqref{eq:rel_Jn1}--\eqref{eq:rel_Jn22}, we can write $p_i$ as $p_i = a^{x-1} (e_{1;1,0}ba) b^{y-1}$, which is zero in $J_l$ by \eqref{eq:rel_Jn31}.
\end{proof}

\subsection{The Cone $\mr{G}_{\Diamond_l}$}
Let $p$ be a path $v_0 \xrightarrow{a_1} v_1 \xrightarrow{a_2} v_2 \xrightarrow{}\cdots\xrightarrow{} v_t$ in a quiver $\Delta$,
where $v_i$'s may not be all distinct.
We can attach a {\em uniserial} representation $S(p)$ of $\Delta$ with its basis elements $m_{v_i}$ labeled by $v_i$.
The action of arrows on $S(p)$ is given by $a_i m_{v_{i-1}} = m_{v_i}$ and $a_i m_{v_j} = 0$ otherwise.

\begin{definition}

A {\em tri-broken path} from $\pm(i;k,j)$ to $\pm(i;j,k)$ denoted by $\op{tp}_{i;j,k}^{\pm}$ is the path
$$e_{i;k,j}^\pm a_\pm^j c_\mp^k b_\pm^j  e_{i;j,k}^\pm.$$
We also define $\op{tp}_{i;i,0}$ and $\op{tp}_{i;0,i}$ as the paths
$e_{i;i,0}a^{i-1}e_{1;1,0} \text{ and } e_{1;0,1}b^{i-1}e_{i;0,i}.$
\end{definition}
\noindent Note that a tri-broken path may not be literally broken thrice in the figure (see Figure \ref{F:Diamond4}).

\begin{example}
The module $S(\op{tp}_{4;1,3}^+)$ has dimension vector
$$\gamma=\e_{4;3,1}+\e_{3;2,1}+\e_{2;1^2}+\e_{1;0,1}+\e_{1;1,0}+\e_{2;1^2}+\e_{3;1,2}+\e_{4;1,3}.$$
Note that the $\gamma(2;1^2)=2$ because the path passes $(2;1^2)$ twice.
\end{example}

For a subpath $p'$ of $p$ of form $v_{s+1}\to v_{s+2}\to \cdots\to v_t$, we use the notation $p/p'$ for the {\em quotient} path $v_{1}\to v_{2}\to \cdots\to v_{s}$.
Note that there is one-to-one correspondence between subpaths of this form and subrepresentations of $S(p)$.
Conversely for $p''=v_{1}\to v_{2}\to \cdots\to v_{s}$, we write $p\setminus p''$ for $p'$.

Let $\Mod J_l$ be the category of (not necessarily finite-dimensional) modules of $J_l$. The algebra $J_l$ is in fact finite-dimensional (Corollary \ref{C:fdJl}), but this is not clear at this stage.
\begin{lemma} \label{L:proj_fp} We have the following projective presentations in $\Mod J_l$.
\begin{align}
\label{eq:projres1} & P_{l;j-1,k+1}^{\pm} \xrightarrow{ab} P_{l;j,k}^{\pm} \to S(\op{tp}_{l;k,j}^{\pm}) \to 0, & \\
\label{eq:projres2} & P_{l;l-1,1}^+ \oplus P_{l;l-1,1}^- \xrightarrow{\sm{ab \\ab}} P_{l;l,0} \to S(\op{tp}_{l;l,0}) \to 0, & \\
& 0\to P_{l;0,l} \to S(e_{l;0,l}) \to 0.
\end{align}
\end{lemma}

\begin{proof} We only prove the first one. The rest is quite obvious.
Without loss of generality, we only deal with the positive cases.
For any path $p$ in $\Diamond_l$, we denote by $[p]$ its equivalence class in $J_l$.
Let $p$ be a path such that $[p]\in P_{i;j,k}^\pm$ is not in the image of $ab$.

It is easy to see that the first arrow of $p$ must be of type $A$ unless $p$ is trivial.
If the second arrow of $p$ is of type $B$, then $[p]$ lies in the image of $ab$.
If the second arrow of $p$ is of type $C$, then by \eqref{eq:rel_Jn4} $p$ is equivalent to a zero path.
Inductively using \eqref{eq:rel_Jn1} and \eqref{eq:rel_Jn4}, we can show that the maximal nonnegative subpath of $p$ is of form $a^x$.
If $p':=p\setminus a^x$ is nontrivial, then $x=j$.

It is easy to see that the first arrow of $p'$ must be of type $C$ unless $p'$ is trivial.
If the second arrow of $p'$ is of type $A$ or $B$, then by \eqref{eq:rel_Jn22} or \eqref{eq:rel_Jn32} we get a contradiction on the statement of the last paragraph.
Inductively using \eqref{eq:rel_Jn1} and \eqref{eq:rel_Jn22} (or \eqref{eq:rel_Jn32}) we can show that the maximal nonpositive subpath of $p'$ is of form $c^z$. If $p'':=p'\setminus c^z$ is nontrivial, then $z=k$.

It is easy to see that the first arrow of $p''$ must be of type $B$ unless $p''$ is trivial.
If the second arrow of $p''$ is of type $A$ or $C$, then by \eqref{eq:rel_Jn21} or \eqref{eq:rel_Jn31} we get a contradiction on the statement of the last paragraph. Inductively using \eqref{eq:rel_Jn1} and \eqref{eq:rel_Jn21} (or \eqref{eq:rel_Jn31}) we can show that $p''$ is of form $b^y$.
The conclusion is that $p$ must be a subpath of the tri-broken path $\op{tp}_{l;k,j}^{\pm}$. This is exactly what we desire.
\end{proof}

\noindent By a similar argument of Lemma \ref{L:proj_fp}, we have that
\begin{lemma} \label{L:injres_fp} We have the following injective presentations $\Mod J_l$.
\begin{align}
\label{eq:injres1} & 0\to S(\op{tp}_{l;j,k}^{\pm}) \to I_{l;j,k}^{\pm} \xrightarrow{(ab)^*}  I_{l;j+1,k-1}^{\pm}, & \\
\label{eq:injres2} & 0\to S(\op{tp}_{l;0,l}) \to I_{l;0,l} \xrightarrow{\sm{(ab)^* & (ab)^*} }  I_{l;1,l-1}^+ \oplus I_{l;1,l-1}^-, & \\
\label{eq:injres3} & 0\to S(e_{l;l,0}) \to I_{l;l,0} \to 0.
\end{align}
\end{lemma}

\begin{corollary} \label{C:fdJl} The Jacobian algebra $J_l$ is finite-dimensional.
\end{corollary}
\begin{proof} First we show by induction on $j$ that all $P_v$ and $I_v$ are finite-dimensional for $v=\pm(l;j,k)$.
$P_{l;0,l}$ and $I_{l;l,0}$ are the 1-dimensional simple module. Then we use \eqref{eq:projres1},\eqref{eq:projres2} and \eqref{eq:injres1},\eqref{eq:injres2} to finish the induction.

Next we claim that each $P_u$ and $I_u$ is finite-dimensional for $u=\pm(l-1;j,k)$.
Since $P_v$ and $I_v$ is finite-dimensional for $v=\pm(l;j,k)$, there are only finitely many paths in $P_u$ (dual paths in $I_u$) passing $\pm(l;j,k)$.
So our claim is equivalent to that $P_u$ and $I_u$ is finite-dimensional
if we restrict to the subquiver where all vertices $\pm(i;j,k)$ with $i<l$.
But this can be showed using a similar argument for $i=l$.

By ``peeling" off the quiver, inductively we can show that each $P_v$ is finite-dimensional so that the Jacobian algebra is finite-dimensional.
\end{proof}

\begin{definition} A vertex $v$ is called {\em maximal} in a representation $M$ if all its subrepresentations are not supported on $v$.
\end{definition}

\begin{lemma} \cite[Lemma 6.5]{Fs1} \label{L:hom=0} Suppose that a representation $T$ contains a maximal vertex. Let $M=\Coker(\g)$, then $\Hom_J(M,T)=0$ if and only if $\g(\dimbar S)\geq 0$ for all subrepresentations $S$ of $T$.
\end{lemma}

\noindent To simplify the notation, for a frozen vertex $v=\pm(l;j,k)$ (resp. $(l;0,l)$~and~$(l;l,0)$),
we set $T_v:=T_{l;j,k}^\pm:=S(\op{tp}_{l;j,k}^\pm)$ (resp. $T_{l;0,l}:=S(\op{tp}_{l;0,l})$ and $T_{l;l,0}:=S(e_{l;l,0})$).
We note that every $T_v$ contains a maximal vertex.

\begin{lemma} \label{L:TIequi} Let $M=\Coker(\g)$, then $\Hom_J(M,T_v)=0$ for each frozen $v$ if and only if $\Hom_J(M,I_v)=0$ for each frozen $v$.
\end{lemma}

\begin{proof}  Since each subrepresentation of $T_v$ is also a subrepresentation of $I_v$, one direction is clear. Conversely, let us assume that $\Hom_J(M,T_v)=0$ for each frozen $v$. We prove that $\Hom_J(M,I_v)=0$ by induction on $k$.
For $k=0$, we have that $T_v=I_v$.
Now suppose that it is true for $k=n-1$, that is, $\Hom_J(M,I_{l;j+1,n-1}^\pm)=0$.
By Lemma \ref{L:injres_fp}, $\Hom_J(M,I_{l;j,n}^\pm)=0$ is equivalent to $\Hom_J(M,T_{l;j,n}^\pm)=0$.
\end{proof}

Let the notation $\sum_{v\in p} \g(v)$ stand for
$\g(v_0)+\g(v_1)+\cdots+ \g(v_t)$ if $p$ is the path $v_0 \xrightarrow{a_1} v_1 \xrightarrow{a_2} v_2 \xrightarrow{}\cdots\xrightarrow{} v_t$.
We define a cone $\mr{G}_{\Diamond_l}\subset \mb{R}^{\Diamond_l^0}$ by $\sum_{v\in p} \g(v) \geq 0$ for the following paths $p$.
\begin{enumerate}
\item all strict subpaths of $\op{tp}_{l;j,k}^\pm$,
\item all subpaths of $\op{tp}_{l;0,l}$,
\item and the trivial path $e_{l;l,0}$.
\end{enumerate}

\begin{theorem} \label{T:LP_Gl} The set of lattice points $\mr{G}_{\Diamond_l}\cap \mb{Z}^{\Diamond_l^0}$ is exactly $G(\Diamond_l,W_l)$.
Moreover, all the defining conditions of $\mr{G}_{\Diamond_l}$ are essential.
\end{theorem}

\begin{proof} Due to Lemma \ref{L:hom=0} and \ref{L:TIequi}, it suffices to show that $\mr{G}_{\Diamond_l}$ is defined by $\g(\dimbar S)\geq 0$ for all subrepresentations $S$ of $T_v$ and all $T_v$.
We notice that these defining conditions are the union of the defining conditions of
\hbox{$\mr{G}_{\Diamond_l}$~and~$\g(\dimbar T_{l;j,k}^\pm)\geq 0$.}
But the latter conditions are redundant because
$\dimbar T_{l;j,k}^\pm = \e_{l;k,j}^\pm + (\dimbar T_{l;j,k}^\pm-\e_{l;k,j}^\pm)$ and
$\dimbar T_{l;j,k}^\pm-\e_{l;k,j}^\pm$ is the dimension vector of a subrepresentation of $T_{l;j,k}^\pm$.

For the last statement, we notice that each path $p$ in the defining conditions contains a unique frozen vertex. So if one defining condition is a positive combination of others, their paths must share a common frozen vertex.
But this is clearly impossible.
\end{proof}


\begin{definition} Given a weight configuration $\bs{\sigma}$ of a quiver $\Delta$ and a convex polyhedral cone $\mr{G} \subset \mb{R}^{\Delta_0}$,
we define the (not necessarily bounded) convex polytope $\mr{G}(\sigma)$ as $\mr{G}$ cut out by the hyperplane sections $\g \bs{\sigma} = \sigma$.
\end{definition}

\section{Cluster Structure in $\SI_{\beta_l}(K_{l,l}^2)$} \label{S:CS}

\subsection{Initial Seeds}
Let $s_{i;j,k}^{\pm}$ be the Schofield's semi-invariant function $s(f_{i;j,k}^{\pm})$.
We consider $\mc{S}_l := \{s_{i;j,k}^{\pm} \mid \pm(i;j,k)\in \Diamond_l^0 \}$ as our choice of the initial cluster.

\begin{lemma} \label{L:2arm} A general representation $M$ in $\Rep_\bl(K_{l,l}^2)$ has a following representative. Its matrices on the negative arm are $\sm{I_k & 0}$; on the positive arm are $\sm{I_k \\ 0}$; on the arrow $a_1$ is $I_l$,
where $I_k$ is a $k\times k$ identity matrix for $k=1\dots,l-1$ and $0$ is a zero column vector.

Moreover, the evaluations of $M$ on $s_{i;j,k}$ and $s_{j,k;i}$ are the $[j+1,i]\times[1,i-j]$-minor and $[1,i-j]\times [j+1,i]$-minor of $M(a_2)$.
Here the interval notation $[i,j]$ means the $i$-th up to $j$-th rows (columns).
\end{lemma}

\begin{proof} If we forget the arrow $a_2$, then we get a quiver of type $A_{2l}$ with the dimension vector shown below
$$\vcenter{\xymatrix@C=5ex{
1\ar[r] & 2\ar[r] & \cdots \ar[r] & l\ar[r]  \ar[r] & l \ar[r] & \cdots \ar[r] & 2 \ar[r] & 1
}}$$
This dimension vector decomposes canonically as $\sum_{i=0}^{l} \bold{1}_i$, where
$$\bold{1}_i=(\underbrace{0,\cdots, 0}_{i-1}, {1,\cdots,1},\underbrace{0,\cdots, 0}_{i-1}).$$
This means that a general representation of that dimension vector is a direct sum of general representations of dimension $\bold{1}_i$,
which is exactly what we desired.
The second statement can be easily verified by elementary matrix calculation.
\end{proof}

We define a restriction map $\Rep_\bl(K_{l,l}^2) \to \Rep_{\beta_{l-1}}(K_{l-1,l-1}^2)$ sending $M$ to $\br{M}$ defined by
\begin{align*} & \br{M}(b_j) = M(b_j) & j&=1,2,\dots,l-2;\\
& \br{M}(a_i) = M(b_{-(l-1)})M(a_i)M(b_{l-1}) & i&=1,2.
\end{align*}
In the language of \cite{Fs2}, the restriction map is the composition of two {\em vertex removals} at $-l$ and $l$.
In particular, it follows from \cite{Fs2} that
\begin{lemma} \label{L:embedding} The restriction map $\Rep_\bl(K_{l,l}^2) \to \Rep_{\beta_{l-1}}(K_{l-1,l-1}^2)$ induces an embedding
$\iota: \SI_{\beta_{l-1}}(K_{l-1,l-1}^2) \hookrightarrow \SI_\bl(K_{l,l}^2)$.
\end{lemma}

\begin{lemma} \label{L:ag-independent} The semi-invariants $s_{i;j,k}^{\pm}$ in $\mc{S}_l$ are algebraically independent.
\end{lemma}

\begin{proof} We consider a representation $M\in \Rep_{\beta_l}(K_{l,l}^2)$ whose matrices on the negative arm are $\sm{I_k & 0}$; on the postive arm are $\sm{I_k \\ 0}$; on the arrows $a_1,a_2$ are the following block matrices
\begin{align*} &
M(a_1)=\begin{pmatrix}
I_{l-1} & X_c \\
X_r & x_{l,l} \end{pmatrix} &
M(a_2)=\begin{pmatrix}
A & 0_c \\
0_r & x_0 \end{pmatrix},
\end{align*}
where $A$ is an $l-1\times l-1$ matrix in general position, $X_r$ is the row vector in generic variables $x_{l,1},\dots,x_{l,l-1}$,
$X_c$ is the column vector in generic variables $x_{1,l},\dots,x_{l-1,1}$, $0_r, 0_c$ are zero vectors of the same size as $X_r, X_c$,
and $x_0$ is another generic variable.
Then according to Lemma \ref{L:2arm}, $s_{l;j,k}(M)$ is easily seen to be a linear function in $x_{l,1}, \dots, x_{l,j}$ and $x_0$.
Similarly, $s_{j,k;l}(M)$ is a linear function in $x_{1,l}, \dots, x_{j,l}$ and $x_0$.
In particular, we see that the set $\{s_{l;j,k}^{\pm}\}_j$ is algebraically independent.

We will finish the proof by induction on $l$. If $l=1$, the statement is clear. Now we assume that the statement is true for $l=n-1$.
In view of Lemma \ref{L:embedding}, we can view $\mc{S}_{n-1}$ as a subset in $\mc{S}_n$.

Now suppose that $\mc{S}_{n} = \{s_{n;j,k}^{\pm}\}_j \cup \mc{S}_{n-1}$ are algebraically dependent, that is, there is a non-zero polynomial $p\in k[y_1,y_2,\dots,y_{2n},\dots,y_{(n^2+n)}]$ such that
$$p\big( s_{n;n,0}^{\pm},\dots,s_{n;0,n}^{\pm}, \mc{S}_{n-1} \big)=0.$$
Its total degree on $y_1,\dots,y_{2n}$ must be strictly positive because elements in $\mc{S}_{n-1}$ are algebraically independent.
We evaluate $p$ at the representation $M$.
Since $\{s_{n;j,k}^{\pm}\}_j$ are algebraically independent, we conclude that there is some function $r\in R$ vanishes at $M$, where $R\subset \SI_{\beta_{n-1}}(K_{n-1,n-1}^2)$ is the subring generated by $\mc{S}_{n-1}$.
By Lemma \ref{L:2arm} and \ref{L:embedding}, the restriction of $M$ on the subquiver $K_{n-1,n-1}^2$ is a general representation of dimension $\beta_{n-1}$.
But a general representation in $\Rep_{\beta_{n-1}}(K_{n-1,n-1}^2)$ is stable, so there is no such $r$. We get a contradiction.
\end{proof}


\subsection{Initial Exchanges}

Recall the weight configuration $\bs{\sigma}_l$ for $\Diamond_l$.
For any mutable vertex $u=\pm (i;j,k)$, we set ${\bar\f}_u = \sum_{u\to v} \bs{\sigma}_l (v) = \sum_{v\to u} \bs{\sigma}_l (v)$.
We have that
\begin{align*}
\bar\f_u &= \pm\left(-\e_{\mp(i-1)} - \e_{\mp i} - \e_{\mp(i+1)} - \e_{\pm(j+1)} + \e_{\pm j} + \e_{\pm(j-1)} + \e_{\pm(k+1)} + \e_{\pm k} + \e_{\pm(k-1)} \right)
\intertext{\quad if $u$ is not horizontal;}
\bar\f_u &= \pm\left(-\e_{\mp(i-1)} - \e_{\mp i} - \e_{\mp(i+1)} - \e_{\mp 1} + \e_{\pm(i+1)} + \e_{\pm i} + \e_{\pm(i-1)} + \e_{\pm 1} \right)\\
& \ \text{if $u$ is horizontal.}
\end{align*}
\noindent As before, we use the convention that $\e_{0}=0$ and $P_0$ is the zero vector space.
If we mutate $(\Diamond_l,\bs{\sigma}_l)$ at $u$, then we get a corresponding mutated weight vector $\f_u'$:
\begin{align*}
\f_u' &= \pm\left(-\e_{\mp(i-1)} - \e_{\mp(i+1)} - \e_{\pm(j+1)} + \e_{\pm(j-1)} + \e_{\pm(k+1)} + \e_{\pm(k-1)} \right); \\
\f_u' &= \pm\left(-\e_{\mp(i-1)} - \e_{\mp(i+1)} - \e_{\mp 1} + \e_{\pm(i+1)} + \e_{\pm(i-1)} + \e_{\pm 1} \right).
\end{align*}

For $\f_u'$ with $u$ positive, we associate a corresponding projective presentation $f_u'$:
\begin{align*}
 P_{j+1}\oplus P_{j-1}\oplus P_{k+1} \oplus P_{k-1}  &
 \xrightarrow{\sm{p_{-(i+1),j+1}^1 & 0 \\ 0 & p_{-(i-1),j-1}^1 \\ p_{-(i+1),k+1}^2 & p_{-(i-1),k+1}^2 \\ 0 & p_{-(i-1),k-1}^2 } }
 P_{-(i+1)} \oplus P_{-(i-1)}, \\ %
 P_{i+1}\oplus P_{i-1}\oplus P_{1}   &
 \xrightarrow{\sm{p_{-(i+1),i+1}^1 & p_{-(i-1),i+1}^1 & p_{-1,i+1}^2 \\ p_{-(i+1),i-1}^1 & 0 & 0 \\ p_{-(i+1),1}^2 & 0 & 0 } }
 P_{-(i+1)} \oplus P_{-(i-1)} \oplus P_{-1}.%
\end{align*}
If $u=-(i;j,k)$ is negative, we associate a projective presentation $f_u':=-f_{-u}'$,
where the first $-$ is the map defined in \eqref{eq:-map}, and $-u=(i;j,k)$.

Finally we can associate a semi-invariant function $s_u'$ for each mutable vertex $u$
\begin{align*}
s_u'&=(-1)^{jk}s(f_u') && \text {$u$ is not horizontal,} \\
s_u'&=(-1)^{i}s(f_u') && \text {$u$ is horizontal.}
\end{align*}

\begin{lemma} \label{L:Uinv'} For $u=\pm(i;j,k)$ with $j>k+1$, $f_u'$ is coherently $U$-invariant so $s_u'$ is $U$-invariant.
\end{lemma}

\begin{proof} By the symmetry we only need to prove for $u$ positive.
It is easy to see that $\sm{1 & u \\ 0 & 1}\cdot f_u' = \sm{e_{j+1} &0 &up_{j+1,k+1} & 0 \\ 0 &e_{j-1} &up_{j-1,k+1} & up_{j-1,k-1} \\ 0 &0 &e_{k+1} &0\\ 0 &0 &0 & e_{k-1} } f_u'$.
Then the assignment $\sm{1 & u \\ 0 & 1}\mapsto \left(\sm{e_{j+1} &0 &up_{j+1,k+1} & 0 \\ 0 &e_{j-1} &up_{j-1,k+1} & up_{j-1,k-1} \\ 0 &0 &e_{k+1} &0\\ 0 &0 &0 & e_{k-1} }, \sm{e_{-(i+1)} & 0 \\ 0 & e_{-(i-1)}} \right)$ defines the algebraic group homomorphism $\varphi$ as in Definition \ref{D:Ginv}.
Then our claim follows from Lemma \ref{L:Ginv}.
\end{proof}

\begin{lemma}  \label{L:su_irreducible} Any $s_u'$ is an irreducible polynomial.
\end{lemma}

\begin{proof} According to Lemma \ref{L:irreducible}, it suffices to show that $\f_u'$ is indivisible and extremal in $\cone{l}$.
The first case is obvious because we observed that any $\sigma\in \cone{l}$ must satisfy $\sgn \sigma(i) = \sgn i$.
For the second case, we observe that any factor of $s_u'$ is also semi-invariant because of \cite[Theorem 3.1]{PV}.
We restrict $s_u'$ on the subvariety
$$X:=\{M\in \Rep_{\bl}(K_{l,l}^2) \mid M(b_{-k})= \sm{I_k & 0}, M(a_1)=I_l, M(b_{k})=\sm{I_k \\ 0}\}.$$
By Lemma \ref{L:2arm}, $\GL_{\bl}\cdot X$ is dense in $\Rep_{\bl}(K_{l,l}^2)$.
We conclude that $s_u'$ is irreducible if and only if its restriction on $X$ is irreducible.
But $s_u'\mid_X=a_{1,i+1}a_{i+1,1} + a_{1,i}a_{i,1}$, which is clearly irreducible.
\end{proof}

\begin{lemma} \label{L:s'}
We have the following relations
\begin{align*}
& s_{i;j,k}^{\pm} (s_{i;j,k}^{\pm})'=s_{i-1;j-1,k}^\pm s_{i;j+1,k-1}^\pm s_{i+1; j,k+1}^\pm + s_{i-1;j,k-1}^\pm s_{i;j-1,k+1}^\pm s_{i+1; j+1,k}^\pm,\\
& s_{i;i,0} (s_{i;i,0})' = s_{i-1;i-1,0} s_{i,1;i+1} s_{i+1;i,1} + s_{i-1,1;i} s_{i;i-1,1} s_{i+1; i+1,0}.
\end{align*}
\end{lemma}

\begin{proof} Due to symmetries, we only need to prove the first one when $\pm$ takes $+$ and $j\geq k$.
We first assume that $j\geq k+2$, i.e., $u$ is a general vertex. Let
$$F_0=s_{i;j,k}^{} (s_{i;j,k}^{})', F_1=s_{i-1;j-1,k} s_{i;j+1,k-1} s_{i+1; j,k+1}, F_2=s_{i-1;j,k-1} s_{i;j-1,k+1} s_{i+1; j+1,k}.$$
We notice that according to Lemma \ref{L:level1} and \ref{L:Uinv'}, each $F_i$ is $U$-invariant with $(\sigma;\lambda)$-degree equal to $(\bar\f_u; 3j,3k)$.
By Corollary \ref{C:dimUinv} and Lemma \ref{L:g=2}, we have that
$$\dim \SI_\bl(K_{l,l}^2)_{\sigma;\lambda}^U=g_{(i+1,i,i-1),(j+1,j,j-1,k+1,k,k-1)}^{(3j,3k)}=2.$$
Hence, we must have that $a_0F_0 = a_1F_1+a_2F_2$, for some $a_0, a_1, a_2\in k$.
To determine $a_i$, we consider a special representation $M$, whose matrices on the negative arm are $\sm{I_k & 0}$; on the positive arm are $\sm{I_k\\0}$; and on the arrow $a_1$ is $I_l$ as in Lemma \ref{L:2arm}.
The entries of matrix $M(a_2)$ are all zeros except for the diagonal of the $[j-1,i]\times [1,i-j+1], [j,i+1]\times[1,i-j+1]$, and $[j+1,i+1]\times [1,i-j]$-submatrix.
The entries on the first and third diagonals are all $1$'s;
The entries on the second diagonal are all $0$'s except that $a_{j+1,1}=1$ and $a_{i,k},a_{i+1,k+1}$ generic.
Then $s_{i-1;j-1,k}(M)$ and $s_{i+1; j+1,k}(M)$ (resp. $s_{i;j+1,k-1}(M)$ and $s_{i;j-1,k+1}(M)$) are given by the determinants of lower (upper) triangular matrices with diagonal all $1$'s.
It can be easily verified by the elementary matrix calculation that
\begin{align*}
s_{i+1; j,k+1}(M) &= (-1)^q (a_{i,k}a_{i+1,k+1}-ra_{i+1,k+1}-1), \\
s_{i-1;j,k-1}(M) &= (-1)^{q}, \\
s_{i;j,k}(M) &= (-1)^{q}(a_{i,k}-r), \\
(s_{i;j,k})'(M) &= a_{i+1,k+1},
\end{align*}
where $k-1=2q+r$ and $r=0,1$.
We get $a_0=a_1=a_2=1$ by solving a linear system.

To show the first relation for $j=k+1,k$, we observe that these two cases are in fact degenerated from the cases when $j\geq k+2$.
To make this more precise, we recall that it suffices to show the identity
evaluating at a general representation.
We take a general representation $M$ with a form as in Lemma \ref{L:2arm}.
We can see that the desired identity for $j=k+1$ can be specialized from the proven identity for $j'=j+1,k'=k$ by setting the first row of $M(a_2)$ to be zeros.

The second relation can be easily verified by evaluating at a general representation $M$ as in Lemma \ref{L:2arm}.
\begin{align*}
s_{i;i,0} (s_{i;i,0})' (M) &= 1\cdot (a_{1,i+1}a_{i+1,1} + a_{1,i}a_{i,1}),\\
s_{i-1;i-1,0} s_{i,1;i+1} s_{i+1;i,1}(M) &= 1\cdot a_{1,i+1}\cdot a_{i+1,1},\\
s_{i-1,1;i} s_{i;i-1,1} s_{i+1; i+1,0}(M) &= a_{1,i} \cdot a_{i,1} \cdot 1.
\end{align*}

\end{proof}
\noindent From the second relation and the symmetry, we also get that
$$s_{i;0,i} (s_{i;0,i})' = s_{i-1;0,i-1} s_{1,i;i+1} s_{i+1;1,i} + s_{1,i-1;i} s_{i;1,i-1} s_{i+1;0,i+1}.$$
Moreover by setting $i=1$ and our convention, we get that
$$s_{1;1,0} (s_{1;1,0})' = s_{1^2;2} s_{2;1^2} + s_{1;0,1}^2 s_{2; 2,0} \ \text{ and }\  s_{1;0,1} (s_{1;0,1})' = s_{1^2;2} s_{2;1^2} + s_{1;1,0}^2 s_{2;0,2}.$$

\subsection{Cluster Structure}

\begin{proposition} \label{P:contain} The upper cluster algebra $\br{\mc{C}}(\Diamond_l,\mc{S}_l;\wtd{\bs{\sigma}}_l)$
is a $\wtd{\sigma}$-graded subalgebra of $\SI_\bl(K_{l,l}^{2})$.
\end{proposition}

\begin{proof} Due to Lemma \ref{L:ag-independent} and \ref{L:WC}, such an assignment does define a graded cluster algebra $\mc{C}(\Diamond_l,\mc{S}_l;\wtd{\bs{\sigma}}_l)$.
By \cite[Theorem 3.17]{PV}, any $\SI_\bl(K_{l,l}^{2})$ is a UFD.
Due to Lemma \ref{L:su_irreducible} and \ref{L:s'}.(2), we can apply Lemma \ref{L:RCA} and get the containment $\br{\mc{C}}(\Diamond_l,\mc{S}_l)\subseteq \SI_\bl(K_{l,l}^{2})$.
By our construction and Lemma \ref{L:WT-var}, the grading $\wtd{\bs{\sigma}}_l$ is consistent with the $\wtd{\sigma}$-grading of $\SI_\bl(K_{l,l}^{2})$.
\end{proof}
\noindent We will see in Remark \ref{r:strictupper} that the upper cluster $\br{\mc{C}}(\Diamond_l)$ strictly contains the cluster algebra $\mc{C}(\Diamond_l)$.

\begin{theorem} \label{T:CS} The semi-invariant ring $\SI_\bl(K_{l,l}^{2})$ is the graded upper cluster algebra $\br{\mc{C}}(\Diamond_l,\mc{S}_l;\wtd{\bs{\sigma}}_l)$. Moreover, the generic cluster character maps the lattice points in $\mr{G}_{\Diamond_l}$ (bijectively) onto a basis of $\br{\mc{C}}(\Diamond_l)$.
\end{theorem}

\noindent We put off the proof to the last section. It is a rather complicated combinatorial proof using several results from \cite{Fs1} and \cite{Fs2}.
However, we still suggest readers carefully reading it because it tells the story of how we find the cluster structure.

\section{Consequences of Cluster Structures} \label{S:CCS}

Our first result says that the weighted counting of $\g$-vectors gives $2$-truncated Kronecker products.
Let $\bs{\lambda}_l$ be the extended part of the weight configuration $\bs{\wtd\sigma}_l$, that is,
$\bs{\lambda}_l(\pm(i;j,k))=(j,k)$.
\begin{proposition} \label{P:2KP} Let $\sigma$ be a weight of $K_{l,l}^2$, then
\begin{equation}
\sum_{\g \in \mr{G}_{\Diamond_l}(\sigma) \cap \mb{Z}^{\Diamond_l^0}} \b{z}^{\g \bs{\lambda}_l} = \left[ s_{\mu(\sigma)} * s_{\nu(\sigma)} \right]_2.
\end{equation}
\end{proposition}

\begin{proof} By Theorem \ref{T:CS}, the generic character maps the lattice points in $\mr{G}_{\Diamond_l}(\sigma,\lambda)$ bijectively to a basis in $\SI_\bl(K_{l,l}^2)_{\sigma,\lambda}$.
Applying the character map $\chi$ to the equality in Lemma \ref{L:SI(Kml)},
we get our desired result.
\end{proof}

\noindent Note that by the symmetry we have that $|\mr{G}_{\Diamond_l}(\sigma,\lambda)\cap \mb{Z}^{\Diamond_l^0}| = |\mr{G}_{\Diamond_l}(\sigma,\omega \lambda)\cap \mb{Z}^{\Diamond_l^0}|$, though this is not clear from the defining condition of $\mr{G}_{\Diamond_l}$.
Recall our definition that $\omega\lambda=(\lambda(2),\lambda(1))$ but $\lambda^\omega=(\lambda(1)+1,\lambda(2)-1)$.

\begin{theorem} \label{T:KC} Let $\sigma$ be a weight of $K_{l,l}^2$ and $\lambda$ be a partition of length $\leq 2$. Then
$$g_{\mu(\sigma),\nu(\sigma)}^\lambda = |\mr{G}_{\Diamond_l}(\sigma,\lambda)\cap \mb{Z}^{\Diamond_l^0}| - |\mr{G}_{\Diamond_l}(\sigma,\lambda^\omega)\cap \mb{Z}^{\Diamond_l^0}|.$$
\end{theorem}

\begin{proof} This follows immediately from Corollary \ref{L:SI(Kml)2} and Theorem \ref{T:CS}.
\end{proof}

\begin{example} We knew that the semi-invariant ring $\SI_{\beta_2}(K_{2,2}^2)$ is the upper cluster algebra associated to the quiver $\Diamond_2$
$$\ckrontwotwotwo$$
Let us use Proposition \ref{P:2KP} to compute $[s_{(2,1)}*s_{(2,1)}]_2$.
In this case, we consider the weight $\sigma=(-1,-1,1,1)$.
We find that there are six lattice points in the $\g$-vector cone $\mr{G}_{\Diamond_2}$.
We list them with their $\lambda$-weights on the right. Of course, by the symmetry we only need half of them.
\begin{align*}
& \e_{1;1,0} + \e_{2;2,0}  & (3,0), &\qquad\qquad \e_{1;0,1} + \e_{2;0,2} &  (0,3); \\
& \e_{1;0,1} + \e_{2;2,0} & (2,1), &\qquad\qquad \e_{1;1,0} + \e_{2;0,2} & (1,2); \\
& \e_{2;0,2} + 2\e_{1;1,0} - \e_{1;0,1} & (2,1), &\qquad\qquad \e_{2;1^2} + \e_{1^2;2} - \e_{1;1,0} & (1,2).
\end{align*}
So $[s_{(2,1)}*s_{(2,1)}]_2 = x^{(3,0)}+2x^{(2,1)}+2x^{(1,2)}+x^{(0,3)} = s_{(3)}+s_{(2,1)}$.

Now we use Theorem \ref{T:KC}. We see that $g_{(2,1),(2,1)}^{(3)} = 1$ and $g_{(2,1),(2,1)}^{(2,1)}=2-1=1$. It is clear that Theorem \ref{T:KC} is more efficient than Proposition \ref{P:2KP}.
\end{example}


\begin{example} \label{ex:l=2}
In this example, we study concretely the $U$-action on the initial cluster.
To simplify the notation, we relabeled the vertices of $\Diamond_2$ as follows.
$$\ckrontwotwotwonum$$

The cluster variables $x_1,x_3,x_4,x_5$ are four known $U$-invariants.
Let us compute the $U$-action on the other two.
We abuse $u$ for the unipotent element $\sm{1 & u \\ 0 & 1}$.
According to the standard highest weight theory, if $x\in V_\lambda$,
then $u\cdot x = \sum_{k\geq 0} u^k x_k$, where $x\in V_{\lambda+k(1,-1)}$.
By explicit calculation, we find that
\begin{align*}
& u\cdot x_2=x_2 + ux_1,\\
& u\cdot x_6=x_6 + uC(\e_2^\circ) +u^2x_5,
\end{align*}
where $\e_2^\circ:=\e_6+\e_1-\e_2$ and $C_{W_2}(\e_2^\circ)=(x_2^2x_5+x_1^2x_6+x_3x_4)x_1^{-1}x_2^{-1}$.
Using this we can easily check that
$w_1:=x_2'-x_2x_5$ and $w_2:= C_{W_2}(2\e_2^\circ)-4x_5x_6$ are also $U$-invariant.
The latter is the well-known {\em Cayley's hyperdeterminant}.
Moreover, they satisfy $x_1^2w_2 = w_1^2 + 4x_3x_4x_5$.
Note that $w_1$ and $w_2$ have well-defined $\g$-vectors:
\begin{align*}
w_1&=x_6x_1^2x_2^{-1}(1+ y_2- y_1y_2),\\
w_2&=(x_6x_1x_2^{-1})^2(1+2y_2-2y_1y_2+y_2^2+2y_1y_2^2+y_1^2y_2^2).
\end{align*}
Their $\g$-vectors are $\e_2':=\e_6+2\e_1-\e_2$ and $2\e_2^\circ$.
Since $g_{(2,1),(2,1)}^{(2,1)}=g_{(2,2),(2,2)}^{(2,2)}=1$, they spanned their corresponding weight spaces.
Note that $g_{(1,1),(1,1)}^{(1,1)}=0$, so the Kronecker coefficients are not {\em saturated}.
\end{example}

\begin{problem} Does any $\wtd{\sigma}$-graded $U$-invariant function in $\SI_\bl(K_{l,l}^2)$ have a well-defined $\g$-vector?
\end{problem}


If a $\wtd{\sigma}$-graded $U$-invariant function does have a well-defined $\g$-vector,
we will see that it must lie on the {\em facet} (i.e., codimension-1 face) of $\mr{G}_{\Diamond_l}$ defined by $\sum_{i=1}^l \g(i;0,i) =0$. This is really a facet due to the second statement of Theorem \ref{T:LP_Gl}.
We define a (full-dimensional) subcone $\mr{G}_{\Diamond_l}'$ of the facet by adding one additional condition that $\g\bs{\lambda}_l(1) - \g\bs{\lambda}_l(2) \geq 0$.
This condition ensures that the $\lambda$-weights of $\g$-vectors in $\mr{G}_{\Diamond_l}'$ are dominant, so it is clearly necessary.

\begin{proposition} \label{P:gUinv} All possible $\g$-vectors of $U$-invariants span a full-dimensional subcone in $\mr{G}_{\Diamond_l}'$.
\end{proposition}

\begin{proof} Let $R:=\SI_\bl(K_{l,l}^2)$ and $X=\Spec R$, then we have that $\dim X = \dim_{\mb{R}} \mr{G}_{\Diamond_l}$.
Let $G^U$ be the set of all possible $\g$-vectors in the $U$-invariants. We explained in Section \ref{S:CC} that $G^U$ is a semigroup.
So by Lemma \ref{L:independent} the degree of the Hilbert polynomial of $R^U$ is no less than the dimension of the cone $\mb{R} G^U$.
We first claim that $G^U$ cannot span a full-dimensional subcone of $\mr{G}_{\Diamond_l}$.
The quotient $Q(R^U)$ of the integral $U$-invariants is contained in the rational $U$-invariants $k(X)^U$.
Since the action of the $1$-dimensional group $U$ is non-trivial, the transcendence degree of $k(X)^U$ is one less than $\dim X$ (\cite[corollary of Theorem 2.3]{PV}).
Our first claim follows.

Now it suffices to find $(\dim_{\mb{R}} \mr{G}_{\Diamond_l} - 1)$ linearly independent $\g$-vectors in $G^U$ satisfying $\sum_{i=1}^l \g(i;0,i) =0$.
We have seen in Example \ref{ex:l=2} that $\e_{2;0,2}+2\e_{1;1,0}-\e_{1;0,1}$ lies in $G^U$.
Besides this one, we consider the following list of $\g$-vectors.
\begin{align*}
& (i-2)\e_{i;0,i}^{\pm} + \e_{2;1^2}^{\pm} + \sum_{n=3}^i (\e_{n;n-1,1}^{\pm} - \e_{n-1;0,n-1}^{\pm}) & &\text{for $3\leq i\leq l$},\\
& (i-2j+1) \e_{i;j,i-j}^{\pm} + \sum_{n=0}^{i-2j-1} (\e_{2j+n;j+n+1,j-1}^{\pm}-\e_{2j+n;j,j+n}^{\pm}) & &\text{for $3\leq i\leq l,2j< i$}.
\end{align*}

It is easy to verify by Theorem \ref{T:LP_Gl} that these $\g$-vectors are in the cone $\mr{G}_{\Diamond_l}$.
So they correspond to weights orthogonal to $\bl$.
By calculation, those weights are
\begin{align*}
&\mp(i-1)\e_{\mp i}\pm i\e_{\pm 1} \pm(i-2)\e_{\pm i},\\
&\mp(i-2j+1)\e_{\mp i}\pm (i-2j+2)\e_{\pm(i-j)}\pm (i-2j)\e_{\pm(i-1)}.
\end{align*}
By checking with the Euler form of $K_{l,l}^2$, we can verify that they correspond to real Schur roots, and thus correspond to $U$-invariants by Lemma \ref{L:gsreal}.

Finally we observe that together with $\e_{i;j,k}^\pm$ for $j\geq k$, these $\g$-vectors form a linearly independent set because each $\g$-vector has a ``leading term" $a\e_{i;j,k}^{\pm}$ with respect to the lexicographical order.
The index of each leading term corresponds to a vertex of $\Diamond_l$ except for $(1;0,1)$.
\end{proof}

Unfortunately the Kronecker coefficient $g_{\mu(\sigma),\nu(\sigma)}^\lambda$ is {\em not} counted by lattice points in $\mr{G}_{\Diamond_l}'(\sigma,\lambda)$.
For example, the $\g$-vector $\e_{2;0,2}+\e_{1;1,0}-\e_{1;0,1}\in \mr{G}_{\Diamond_l}'$ corresponds to $g_{(1,1),(1,1)}^{(1,1)}$ (see Example \ref{ex:l=2}) but $g_{(1,1),(1,1)}^{(1,1)}=0$.

%

\section{Computation of Some Invariant Rings} \label{S:Inv}

\subsection{Extremal Unimodular Fan}
Let $\mr{K}\subset \mb{R}^d$ be a {\em pointed} rational cone with {\em generator} $v_1,v_2,\dots,v_m\in \mb{Q}^d$,
that is, points in $\mr{K}$ are positive linear combination of $v_1,v_2,\dots,v_m$ and there exists a hyperplane $H$ for
which $H\cap \mr{K} = 0$.
The cone $\mr{K}$ is called {\em simplicial} if it has precisely $d$ linearly independent generators. A simplicial cone is called {\em unimodular} if its generators form a $\mb{Z}$-basis of $\mb{Z}^d$.

A collection $\mc{T}$ of simplicial $d$-dimensional cones is a {\em triangulation} of the $d$-dimensional cone $\mr{K}$ if it satisfies: \begin{enumerate}
\item $\mr{K}=\bigcup_{\mr{S}\in \mc{T}} \mr{S}$.
\item For any $\mr{S}_1,\mr{S}_2\in \mc{T}, \mr{S}_1\cap \mr{S}_2$ is a common face of $\mr{S}_1$ and $\mr{S}_2$.
\end{enumerate}
We say that $\mr{K}$ can be triangulated using no new generators if there exists a triangulation $\mc{T}$ such that the generators of any $\mr{S}\in\mc{T}$ are generators of $\mr{K}$.
It is well-known \cite[Section 2.6]{F} that \begin{enumerate}
\item[$\bullet$] any pointed rational cone can be triangulated into unimodular cones;
\item[$\bullet$] any pointed cone can be triangulated into simplicial cones using no new generators.
\end{enumerate}
However, in general pointed rational cones may not be triangulated into unimodular cones using no new generators.

Let $\mr{G}_{\Delta}$ be the cone generated by all $\g$-vectors of the upper cluster algebra $\br{\mc{C}}(\Delta)$.
It is clear that for any frozen vertex $v\in \Delta_0$, $\e_v$ is a generator of $\mr{G}_{\Delta}$.
We denote by $\mb{F}$ the set $\{\e_v \mid v \text{ frozen}\}$.

\begin{definition} An {\em extremal unimodular fan} of $\br{\mc{C}}(\Delta)$ is a triangulation of $\mr{G}_{\Delta}$ into unimodular cones with generators $\mb{F}\cup \{\g_1,\dots,\g_p \}$, where each $\g_i$ is a generator of $\mr{G}_{\Delta}$.
The set of generators of each unimodular cone is called a {\em fake cluster}.
\end{definition}

We can guarantee neither the existence nor the uniqueness of the extremal unimodular fan of an upper cluster algebra.
But when it exists, we shall see that it is a crucial tool to present an upper cluster algebra with infinitely many clusters.
Here at least we can easily see that the upper cluster algebra is generated by elements of extremal $\g$-vectors if lattice points in $\mr{G}_{\Delta}$ index a basis of $\br{\mc{C}}(\Delta)$.
This idea will be elaborated somewhere else.
We believe that an extremal unimodular fan exists for a huge class of upper cluster algebras including $\br{\mc{C}}(\Diamond_l)$.
\begin{conjecture} Each $\br{\mc{C}}(\Diamond_l)$ has an extremal unimodular fan.
\end{conjecture}

\noindent In some cases, the extremal unimodular fan agrees with the cluster fan.
As a nontrivial example, the cluster algebra $\mc{C}(\Delta_4)$ \cite{Fs1} is one such case.
In general, even if the cluster algebra has only finitely many clusters,
a fake cluster may {\em not} be a cluster.

For any unimodular cone $\mr{S}$ with generators $u_1,\dots,u_d$, the generating function for the set of integer points in $\mr{S}$ is a multiple geometric series
$$f(\mr{S},\b{z})=\sum_{m\in \mr{S}\cap \mb{Z}^d} \b{z}^m = \prod_{i=1}^d (1-\b{z}^{u_i})^{-1}.$$
Using this observation, it is easy to compute the multigraded Hilbert series of the upper cluster algebras.

\subsection{Two Examples}
Let us come back to Example \ref{ex:l=2}.
Recall that $\e_1'=\e_{3} +\e_{4} - \e_{1}, \e_2' =\e_6 +2\e_1 -\e_{2}$ and $\e_2^\circ =\e_6 +\e_1 -\e_{2}$.
We set $x_2^\circ:=C_{W_2}(\e_2^\circ)$.
Reader should be aware that $x_2^\circ$ is {\em not} a cluster variable, because $\e_2^\circ$ is not a {\em real} $\g$-vector in the sense of \cite[Definition 4.6]{DF}.
To simplify notation, we will write $\wtd\sigma_i$ for the weight vector $\wtd{\bs{\sigma}}_2(i)$.

\begin{proposition} {\ }
\begin{enumerate}
\item $\SI_{\beta_2}(K_{2,2}^2)$ is minimally presented by
$$k[x_1,x_2,\dots,x_6,x_2^\circ]/(x_2^2x_5+x_1^2x_6+x_3x_4- x_1x_2x_2^\circ).$$
Its $\sigma$-graded character is equal to
$$(1-\b{z}^{\wtd{\sigma}_3+\wtd{\sigma}_4})(1-\b{z}^{\wtd{\sigma}_2^\circ})^{-1}\prod_{i=1}^6 (1-\b{z}^{\wtd{\sigma}_i})^{-1}.$$
\item $\SI_{\beta_2}(K_{2,2}^2)^U$ is minimally presented by
$$k[w_1,w_2,x_1,x_3,x_4,x_5]/ (w_1^2 + 4x_3x_4x_5 - w_2x_1^2).$$
Its $\wtd\sigma$-graded Hilbert series is equal to
$$(1+\b{z}^{\wtd{\sigma}_2'})\left((1-\b{z}^{\wtd\sigma_1})(1-\b{z}^{\wtd\sigma_3})(1-\b{z}^{\wtd\sigma_4})(1-\b{z}^{\wtd\sigma_5})(1-\b{z}^{2\wtd{\sigma}_2^\circ})\right)^{-1}.$$
\end{enumerate} \end{proposition}

\begin{proof} (1). The cone $\mr{G}_{\Diamond_2}$ admits a triangulation by four unimodular cones.
They are spanned by the unions of $\mb{F}=\{\e_{3},\e_{4},\e_{5},\e_{6} \}$
with the following four pairs of vectors
$$\{ \e_{1}, \e_{2} \}, \{\e_{1}', \e_{2} \},
\{ \e_{1}, \e_2^\circ \}, \{\e_2^\circ, \e_1' \}.$$
By definition, this is an extremal unimodular fan.
In particular, $\SI_{\beta_2}(K_{2,2}^2)$ is generated by the initial cluster together with
$x_1'$ and $x_2^\circ$. The weights of $x_i$ are all extremal in $\Sigma_{\beta_2}(K_{2,2}^2)$, but $x_1' = x_2^\circ x_2-x_1x_6$.
So it is minimally generated by $x_1,x_2,\dots,x_6,x_2^\circ$, and satisfies $x_2^2x_5+x_1^2x_6+x_3x_4=x_1x_2x_2^\circ$. Since $x_2^2x_5+x_1^2x_6+x_3x_4- x_1x_2x_2^\circ$ is irreducible and the Krull dimension of $\SI_{\beta_2}(K_{2,2}^2)$ is $6$,
we conclude that it is the only relation.

Using the extremal unimodular fan and the inclusion-exclusion, we get the generating function for the lattice points in $\mr{G}_{\Diamond_2}$:
\begin{align*}
& \prod_{i=3}^6 (1-\b{z}^{\e_i})^{-1} \Big( (1-\b{z}^{\e_1})^{-1}(1-\b{z}^{\e_2})^{-1} + (1-\b{z}^{\e_1})^{-1}(1-\b{z}^{\e_2^\circ})^{-1} + (1-\b{z}^{\e_1'})^{-1}(1-\b{z}^{\e_2})^{-1} \\
& + (1-\b{z}^{\e_1'})^{-1}(1-\b{z}^{\e_2^\circ})^{-1} - (1-\b{z}^{\e_1})^{-1} - (1-\b{z}^{\e_1'})^{-1} - (1-\b{z}^{\e_2})^{-1} - (1-\b{z}^{\e_2^\circ})^{-1} +1 \Big) \\
& = \prod_{i=3}^6 (1-\b{z}^{\e_i})^{-1} \frac{(1-\b{z}^{\e_1+\e_1'})(1-\b{z}^{\e_2+\e_2^\circ})}{(1-\b{z}^{\e_1'})(1-\b{z}^{\e_2^\circ})}.
\end{align*}
By Theorem \ref{T:CS}, we remain to transfer the $\g$-vector weights to the $\wtd{\bs{\sigma}}$-weights.
We do get the desired graded character because $\e_1' \wtd{\bs{\sigma}}_2= (\e_2+\e_2^\circ)\wtd{\bs{\sigma}}_2$.

(2). Using the Omega package \cite{APR}, we can turn a rational expression for a formal series
$\sum a_\lambda s_\lambda$ into a rational expression for $\sum a_\lambda \b{z}^\lambda$.
Applying to the graded character, we get the desired $\wtd{\bs{\sigma}}$-graded Hilbert series.
Let us consider the subalgebra $R\subseteq \SI_{\beta_2}(K_{2,2}^2)^U$ generated by $\{w_1,w_2,x_1,x_3,x_4,x_5\}$.
Since the last five generators have linearly independent $\wtd{\bs{\sigma}}$-degrees, they are algebraically independent.
Recall from Example \ref{ex:l=2}, they satisfy the relation $r:=w_1^2 + 4x_3x_4x_5 - w_2x_1^2=0$.
Since $r$ is irreducible, we conclude that $R$ is presented by $k[w_1,w_2,x_1,x_3,x_4,x_5]/(r)$.
Now we rewrite the $\wtd{\bs{\sigma}}$-graded Hilbert series as
$$\left(1-\b{z}^{2\wtd{\sigma}_2'}\right)\left((1-\b{z}^{\wtd\sigma_1})(1-\b{z}^{\wtd\sigma_3})(1-\b{z}^{\wtd\sigma_4})(1-\b{z}^{\wtd\sigma_5})(1-\b{z}^{2\wtd{\sigma}_2^\circ})(1-\b{z}^{\wtd{\sigma}_2'})\right)^{-1},$$
which is the same the Hilbert series of $R$.
Hence, we must have that $R=\SI_{\beta_2}(K_{2,2}^2)^U$.
\end{proof}

\begin{remark} \label{r:strictupper} Since the $\sigma$-weight of $x_2^\circ$ is extremal, the non-cluster variable $x_2^\circ$ is contained in any homogeneous generating set of $\SI_{\beta_2}(K_{2,2}^2)$ (in fact $\SI_{\beta_2}(K_{l,l}^2)$).
This provides an example of an acyclic cluster algebra which is strictly contained in its upper cluster algebra.
This is different from the result in \cite{M2} due to the difference in the definition of upper cluster algebras.
Moreover, we believe that $\SI_{\beta_2}(K_{2,2}^2)^U$ (and $\SI_{\beta_l}(K_{l,l}^2)^U$ in general) is {\em not} an (upper) cluster algebra, though we do not have a proof for that.
\end{remark}

As our last example, we want to show readers that the cluster structure still exists even if the two flags have different length, i.e., $l_1\neq l_2$.
\begin{example} Recall from \cite{Dv} that $g_{\mu,\nu}^\lambda=0$ whenever $\ell(\lambda)>\ell(\mu)\ell(\nu)$. So by the $S_3$-symmetry of the Kronecker coefficients, the semi-invariant ring below controls all (full) Kronecker products of $s_\mu * s_\nu$ if $\ell(\mu),\ell(\nu)\leq 2$.

Consider the quiver with dimension vector $(K_{4,2}^2,\beta_{4,2})$ obtained from $(K_{4,4}^2,\beta_4)$ by removing the vertex $4$ and $3$ (see below Theorem \ref{T:tripleflag}).
$$\krontwofourtwo{}{}{1}{2}{3}{4}$$
By Proposition \ref{P:removal}, we find that
the semi-invariant ring $\SI_{\beta_{4,2}}(K_{4,2}^2)$ is the upper cluster algebra of the following quiver. The right one is a relabeling of the left one.
$$\ckrontwofourtwo{\fr{_{4;2^2}}} \qquad\qquad  \ckrontwofourtwonum{\fr{9}}$$

Note that the ice quiver is not connected.
If we delete the frozen vertex $(4;2^2)$, then the associated upper cluster algebra is
the semi-invariant ring for the quiver with dimension vector $(K_{3,2}^2,\beta_{3,2})$:
$$\krontwothreetwo{}{}{1}{2}{3}$$
It suffices to study this semi-invariant ring.
We keep that $\e_1' = \e_4+\e_3-\e_1, \e_2' =\e_6 +2\e_1 -\e_{2},\e_2^\circ = \e_6 + \e_1 -\e_2$, and let
$$\e_3' = \e_8+\e_5+\e_2-\e_3,\e_{31}=\e_8+\e_4-\e_1, \e_{32}=\e_8+\e_5+\e_1-\e_3, \text{ and } \e_{312}=2\e_8+\e_5-\e_3.$$
Using \cite{MPT} we find that together with $\mb{F}$ they constitute generators of the $\g$-vector cone.

\end{example}

\begin{proposition}
$\SI_{\beta_{3,2}}(K_{3,2}^2)$ is minimally generated by
$$\{x_1,x_2,\dots,x_8,x_2^\circ,x_3',x_{31}:=C_W(\e_{31}),x_{32}:=C_W(\e_{32}),x_{312}:=C_W(\e_{312})\}.$$
\end{proposition}

\begin{proof}
Except for $\e_1'$ these vectors all correspond to extremal weights of $\Sigma_{\beta}(K_{3,2}^2)$.
The $\g$-vector cone admits a triangulation by $14$ unimodular cones.
The cones are spanned by the unions of $\mb{F}$ with the following sets of vectors
\begin{gather*}\{ \e_1, \e_2, \e_3\}, \{ \e_1', \e_2, \e_3 \}, \{ \e_1, \e_2^\circ, \e_3 \}, \{ \e_1, \e_2, \e_3' \}, \\
\{ \e_1, \e_2^\circ, \e_{23} \}, , \{ \e_1, \e_{23}, \e_{3}' \} , \{ \e_{13}, \e_2, \e_3' \}, \{ \e_1', \e_2, \e_{13} \}, \{ \e_1', \e_2^\circ, \e_3 \},\\
\{ \e_{321}, \e_{23}, \e_3' \}, \{ \e_{13}, \e_{321}, \e_3' \}, \{ \e_{312}, \e_2^\circ, \e_{23} \}, \{ \e_1', \e_2^\circ, \e_{13} \}, \{ \e_{31}, \e_2^\circ, \e_{321} \}.
\end{gather*}
We have seen in the previous example that $x_1'$ can be generated from others.
We do get the desired minimal generating set.
\end{proof}

\begin{remark}
Using the extremal unimodular fan and the inclusion-exclusion, we can get the generating function for the lattice points in $\mr{G}_{\Diamond_{3,2}}$.
After transferring the weights as in the previous example, we obtain the $\sigma$-graded character of $\SI_{\beta_{3,2}}(K_{3,2}^2)$, which is a rather long expression.
With the help of the Omega package \cite{APR}, we can recover a old result in \cite{PS} on the generating function $\sum g_{\mu,\nu}^\lambda \b{z}^{(\lambda,\mu,\nu)}$, where the summation runs through all triple of partitions $(\lambda,\mu,\nu)$ with $\ell(\lambda)\leq 3$ and $\ell(\mu)=\ell(\nu)\leq 2$.
With a little effort, one can show that the ring $\SI_{\beta_{3,2}}(K_{3,2})^U$ is minimally generated by 13 invariants of $\tilde{\sigma}$-degree equal to
\begin{align*}
\wtd\sigma_1,\wtd\sigma_3,\wtd\sigma_4,\wtd\sigma_5,2\wtd{\sigma}_2^\circ,\wtd\sigma_3',\wtd\sigma_{32}',\wtd\sigma_1+\wtd\sigma_8,\wtd{\sigma}_{312}',\wtd\sigma_2',\wtd\sigma_2^\circ+\wtd\sigma_7,\wtd\sigma_2^\circ+\wtd\sigma_{32},\wtd\sigma_2'+\wtd\sigma_8.
\end{align*}
\end{remark}

\section{Proof of Theorem \ref{T:CS}} \label{S:proof}
By Corollary \ref{C:GCC}, the image of the lattice points in $\mr{G}_{\Diamond_l}(\sigma)$ under the generic character $C_{W_l}$
is linearly independent in $\SI_\bl(K_{l,l}^2)_\sigma$ for any weight $\sigma$.
Due to Proposition \ref{P:contain}, it suffices to show that the dimension of $\SI_\bl(K_{l,l}^2)_\sigma$ is counted by the lattice points in $\mr{G}_{\Diamond_l}(\sigma)$.
We will prove this in three steps.
In step one, we show that the semi-invariant ring of certain (noncomplete) triple flag is an upper cluster algebras, and describe its IQP model.
In step two, we use orthogonal projection to show that the semi-invariant ring of certain triple-flagged Kronecker quiver representations is an upper cluster algebra, and describe its IQP model and $\g$-vector cone.
In step three, by embedding the quiver $K_{l,l}^2$ to the triple-flagged Kronecker quiver,
we show that there are required number of lattice points in $\mr{G}_{\Diamond_l}(\sigma)$.

\subsection*{Step 1} We first recall the main result from \cite{Fs1}.
The {\em complete triple flag} of length $n$ is the quiver with dimension vector $(T_n,\beta_n)$ as indicated below
$$\tripleflag$$
The {\em ice hive quiver $\Delta_n$ of size $n$} is an ice quiver in the plane with $\sm{n+2\\ 2}-3$ vertices arranged in a triangular grid.
We label the vertices as shown in Figure~\ref{fig:hive}.
Note that three vertices $(0,0),(n,0)$ and $(0,n)$ are missing, and all boundary vertices are frozen.
There are three types of arrows: $(i,j)\to (i+1,j), (i,j)\to (i,j-1)$, and $(i,j) \to (i-1,j+1)$.
The quiver $\Delta_n$ can be equipped with a rigid potential $W_n$ (see \cite{Fs1} for detail).

\begin{figure}[h]
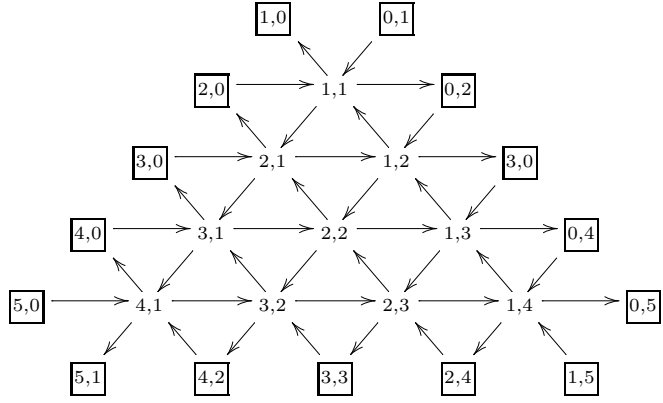

\begin{center}
$\hivesix$
\end{center}
\caption{The quiver $\Delta_6$.}
\label{fig:hive}
\end{figure}

Let $\e_i^a$ be the unit vector in $\mb{Z}^{T_n}$ supported on the $i$-th vertex of the $a$-th arm of $T_n$.
By convention $\e_0^a$ is the zero vector.
We define a (full) weight configuration $(\Delta_n,\bs{\sigma}_n)$ by
the assignment $(i,j)\mapsto \e_{n}-\e_i^1-\e_j^2-\e_k^3$ on the vertices of $\Delta_n$.

\begin{theorem}\cite{Fs1} \label{T:tripleflag} The semi-invariant ring $\SI_{\beta_n}(T_n)$ is the graded upper cluster algebra $\br{\mc{C}}(\Delta_n;\bs{\sigma}_n)$.
Moreover, $(\Delta_n,W_n)$ is a cluster model.
\end{theorem}

Next, we need to recall another result about {\em vertex removal} from \cite{Fs2}.
Let $Q$ be a quiver, and $\beta$ a dimension vector of $Q$.
For a fixed vertex $r\in Q$, let $\br{Q}$ be the quiver such that $\br{Q}_0=Q_0\setminus \{r\}$,
and
$$\br{Q}_1=Q_1\setminus \{a\mid h(a)=r \text{ or } t(a)=r\} \cup \{[ab]\mid h(a)=t(b)=r\},$$
where $t([ab])=t(a)$ and $h([ab])=h(b)$.
Let $\br{\beta}$ be the restriction of $\beta$ on $\br{Q}_0$.

\begin{definition} For an IQP $(\Delta,W)$, {\em deleting} a set of vertices $\b{v}\subset \Delta_0$ is the operation that
we delete all vertices in $\b{v}$ for $\Delta$, and set all incoming and outgoing arrows of $v\in \b{v}$ to be zero in $W$.
We denote the new IQP by $(\Delta_\b{v},W_{\b{v}})$.

For an IQP $(\Delta,W)$, {\em freezing} a set of vertices $\b{u}$ is the operation that
we set all vertices in $\b{u}$ to be frozen, but we do not delete arrows between frozen vertices.
We also keep the same potential $W$. The new quiver is denoted by $\Delta^\b{u}$.
\end{definition}

A set of vertices $\b{u}$ is called {\em attached} to another set of vertices $\b{v}$ if $t(a) \in \b{v}$ (resp. $h(a)\in \b{v}$) whenever $h(a) \in \b{u}$ (resp. $t(a)\in \b{u}$).
For such a pair $(\b{u},\b{v})$, we form a new IQP $(\Delta_{\b{v}}^{\b{u}},W_{\b{v}})$ by deleting $\b{v}$ and freezing $\b{u}$.
We denote by $\bs{\sigma}(\hat{\b{v}})$ the restriction of $\bs{\sigma}$ on $\Delta_{\b{v}}^{\b{u}}$.

\begin{theorem}\cite[Theorem 3.7 and 5.17]{Fs2} \label{T:removal}
Suppose that $\SI_\beta(Q)$ is a naturally graded upper cluster algebra $\uca(\Delta;\bs{\sigma})$, and $B(\Delta)$ has full rank.
Let $r$ be a vertex of $Q$ satisfying $\min\left( \sum_{h(a)=r} \beta(t(a)),\sum_{t(b)=r} \beta(h(b)) \right) \leq \beta(r)$
and that $\bs{\sigma}(v)(r) = 0$ for all $v\in \Delta_0$ except for a set of frozen vertices $\b{v}$, where all $\bs{\sigma}(v)(r)<0$ (or all $>0$).
Then we have that $\SI_{\br{\beta}}(\br{Q})$ is the naturally graded upper cluster algebra $\uca(\Delta_{\b{v}}^{\b{u}};\bs{\sigma}(\hat{\b{v}}))$, where $\b{u}$ is attached to $\b{v}$.

Moreover, if $(\Delta,W)$ is a cluster model, then so is $(\Delta_{\b{v}}^{\b{u}},W_{\b{v}})$.
\end{theorem}

We start with the quiver with dimension vector $(T_{3l},\beta_{3l})$.
We remove the last $l-1$ vertices (i.e., from the $(2l+1)$-th to the $(3l-1)$-th) for each arm.
We denote the resulting quiver with dimension vector by $(T_{2l+1},\br{\beta}_{3l})$.
We repeatedly apply Theorem \ref{T:removal} to those vertices in the order $r = 3l-1,3l-2,\dots,2l+1$ for each arm,
and obtain an IQP denoted by $(\hexagon_l, W_{\hexagon_l})$. Let $\br{\bs{\sigma}}_{3l}$ be the restriction of $\bs{\sigma}_{3l}$ on $\hexagon_l$.
Roughly speaking, $\hexagon_l$ is a hexagonal subquiver of $\Delta_{3l}$ centered at the vertex $(l,l)$ with edge length equal to $l$ arrow length, and all boundary vertices frozen.
When $l=4$, the quiver is displayed in Figure \ref{F:hex4}.

\begin{figure}[h]
\begin{center}
$\hexagonfour$
\end{center}
\caption{The quiver $\hexagon_4$.} \label{F:hex4}
\label{fig:square}
\end{figure}

We summarize our result in step one.
\begin{proposition} The semi-invariant ring $\SI_{\br{\beta}_{3l}}(T_{2l+1})$ is the graded upper cluster algebra $\br{\mc{C}}(\hexagon_l;\br{\bs{\sigma}}_{3l})$.
Moreover, $(\hexagon_l,W_{\hexagon_l})$ is a cluster model (Definition \ref{D:model}).
\end{proposition}

\subsection*{Step 2}
Let $Q$ be a finite quiver without oriented cycles.
For two dimension vector $\alpha,\beta$, we write $\alpha \perp \beta$ if $\Hom_Q(M,N)=\Ext_Q(M,N)=0$ for $M,N$ general in $\Rep_\alpha(Q)$ and $\Rep_\beta(Q)$.
We recall that an {\em exceptional sequence of dimension vector} $\mb{E}:=\{\ep_1,\ep_2,\dots, \ep_n\}$ is a sequence of real Schur roots of $Q$ such that $\ep_i\perp \ep_j$ for any $i<j$.
It is called {\em quiver} if $\innerprod{\ep_j,\ep_i}_Q\leq 0$ for any $i<j$.
It is called {\em complete} if $n=|Q_0|$.
The {\em quiver} $Q(\mb{E})$ of a quiver exceptional sequence $\mb{E}$ is by definition the quiver with vertices labeled by $\ep_i$ and $\ext_Q(\ep_j,\ep_i)$ arrows from $\ep_j$ to $\ep_i$.
According to \cite[Theorem 4.1]{S2}, $\ext_Q(\ep_j,\ep_i)=-\innerprod{\ep_j,\ep_i}_Q$.

Given an exceptional sequence of dimension vector $\mb{F}=\{\vep_1,\vep_2,\dots,\vep_r \}$ with $F_i$ general in $\Rep_{\vep_i}(Q)$,
we consider the (right) {\em orthogonal subcategory} $\mb{F}^\perp$ defined by
$$\mb{F}^\perp:=\{M\in\module(Q) \mid F_i\perp M \text{ for } i=1,2,\dots,r \}.$$
According to \cite{S1}, this abelian subcategory is equivalent to the module category of another quiver denoted by $Q_{\mb{F}}$.
We say that $Q$ {\em projects} to $Q_{\mb{F}}$ through $\mb{F}$.
In this case there is a unique quiver exceptional sequence $\mb{E}=\{\ep_1,\ep_2,\dots, \ep_{|Q_0|-r}\}$ of $Q$ such that
the concatenation of two sequences $(\mb{F},\mb{E})$ is a (complete) exceptional sequence and $Q(\mb{E})\cong Q_{\mb{F}}$.
We refer to \cite{S1} for the construction of $\mb{E}$.

We define a linear isometry (with respect to the Euler forms) $\iota_{\mb{F}}: K_0(Q_{\mb{F}})\to K_0(Q)$ by
\begin{equation} \label{eq:dimemb} \beta' \mapsto \sum \beta'(i)\ep_i.\end{equation}
Conversely we define a linear map $\pi_{\mb{F}}$ left inverse to $\iota_{\mb{F}}$ as the composition $K_0(Q)\to K_0(\mb{F}^\perp) \cong K_0(Q_{\mb{F}})$,
where the first map $\beta_0:=\beta \mapsto \beta_r$ is given by the recursion
$$\beta_{i+1} = \beta_i- \innerprod{\vep_{i+1},\beta_i}_Q \vep_{i+1}.$$
Both maps induce a linear map on the weights, still denoted by $\iota_{\mb{F}}$ and $\pi_{\mb{F}}$.
We write $\beta_{\mb{F}}$ and $\sigma_\mb{F}$ for $\pi_{\mb{F}}(\beta)$ and $\pi_{\mb{F}}(\sigma)$.
To effectively compute $\sigma_\mb{F}$ later, we can use the formula
\begin{equation} \label{eq:wtproj} \sigma_\mb{F}(i) = \innerprod{\sigma, \ep_i}_Q.
\end{equation}

\begin{theorem}[{\cite[Theorem 2.39]{DW2}}] $\SI_{\beta}(Q_{\mb{F}})_{\sigma} = \SI_{\iota_{\mb{F}}(\beta)}(Q)_{\iota_{\mb{F}}(\sigma)}$.
\end{theorem}

It turns out that the quiver $T_{2l+1}$ projects to the following triple-flagged Kronecker quiver $TK_l$
through the exceptional sequence $\mb{F}:=\{\vep_{1,2},\vep_{2,1},\vep_{1,3},\vep_{2,3},\vep_{1,2,3}\}$,
where $\vep_{a,b}= {^\alpha(\e_{3l}-\e_{l}^a-\e_{2l}^b)}$ and $\vep_{1,2,3}={^\alpha(\e_{3l}-\e_{l}^1-\e_{l}^2-\e_{l}^3)}$.
We prove this claim by explicitly putting the real Schur roots in $\mb{E}$ on the quiver of $\mb{E}$.
Readers can easily check that $\mb{F}\perp \mb{E}$ and $\mb{E}$ is a quiver exceptional sequence with the prescribed quiver.
$$\tripleflaggedequi$$
where $\ep_{-l}:={^\alpha(2\e_{3l}-\e_{2l}^1-\e_{2l}^2-\e_{l}^3)}$ and $\ep_l:={^\alpha(\e_{3l}-\e_{2l}^3-\e_{l}^1-\e_{l}^2)}$.
We label the vertices and arms of this quiver as follows
$$\tripleflaggedequidim$$
Let $\beta_l'$ be the dimension vector of $TK_l$ which is $(1,2,,\dots,l)$ along each arm.
We can easily check by \eqref{eq:dimemb} that $\iota_{\mb{F}}(\beta_l') = \br{\beta}_{3l}$.

\begin{figure}[h]
\begin{center}
$\hexagonfourhatnew$
\end{center}
\caption{The quiver $\widehat{\hexagon}_4$.}
\label{F:hathexagon}
\end{figure}

Suppose that $\SI_\beta(Q)$ is an upper cluster algebra $\br{\mc{C}}(\Delta; \bs{\sigma})$ naturally graded by the weights of semi-invariants.
Let $e$ be a vertex of $\Delta$ such that $\bs{\sigma}(e)$ corresponds to a real Schur root $\vep$ of $Q$.
We define the {\em projection} of the weight configuration $(\Delta,\bs{\sigma})$ through $e$ as the pair $(\Delta_e,\bs{\sigma}_e)$, where $\bs{\sigma}_e(v)=\bs{\sigma}(v)_\vep$.

\begin{theorem}\cite[Theorem 3.7 and 5.6]{Fs2} \label{T:project}
Suppose that $\SI_\beta(Q)$ is a naturally graded upper cluster algebra $\uca(\Delta; \bs{\sigma})$, and $B(\Delta)$ has full rank.
If $\bs{\sigma}(e)={^\sigma}\vep$ is real and extremal in $\mb{R}^+\Sigma_\beta(Q)$,
then $\vep$ is Schur and $\SI_{\beta_\vep}(Q_\vep)$ is the naturally graded upper cluster algebra $\uca(\Delta_e; \bs{\sigma}_e)$.

Moreover, if $(\Delta,W)$ is a cluster model, then so is $(\Delta_e,W_e)$.
\end{theorem}

\noindent We note that if $\vep_1,\vep_2,\dots,\vep_r$ is an exceptional sequence with each $^\sigma \vep_i$ extremal in $\Sigma_\beta(Q)$,
then each $^\sigma(\vep_2)_{\vep_1},\dots,{^\sigma}(\vep_{n})_{\vep_1}$ is extremal in $\Sigma_{\beta_{\vep_1}}(Q_{\vep_1})$.
So we can apply Theorem \ref{T:project} to the sequence of vertices $(l,2l),(2l,l),(0,l),(l,0),(l,l)$ of $\hexagon_l$ corresponding to the exceptional sequence $\mb{F}$ of $TK_l$.
We obtain an IQP denoted by $(\widehat{\hexagon}_l, W_{\widehat{\hexagon}_l})$.
It follows that the semi-invariant ring $\SI_{\beta_l'}(TK_l)$ is the graded upper cluster algebra $\br{\mc{C}}(\widehat{\hexagon}_l; \widehat{\br{\bs{\sigma}}}_{3l})$,
where the projected weight configuration $\widehat{\br{\bs{\sigma}}}_{3l}$ is given by
\begin{align*}
\widehat{\br{\bs{\sigma}}}_{3l}(i,j) = \begin{cases} 2\e_l - \e_{2l-(i+j)}^3 -\e_j^1 -\e_i^2 & \text{if }{i,j\leq l,}\\
\e_{-l} - \e_{(i+j)-3l}^{3} -\e_{l-i}^{1}-\e_{l-j}^{2} & \text{if }{i,j\geq  l,} \\
\e_{l} - \e_{2l-(i+j)}^3-\e_j^1-\e_{l-i}^{1} & \text{if }{i>l, i+j\leq 2l,}\\
\e_{l} - \e_{2l-(i+j)}^3-\e_i^2-\e_{l-j}^{2} & \text{if }{j>l, i+j\leq 2l,}\\
\e_{l}+\e_{-l} - \e_{(i+j)-3l}^3-\e_j^1-\e_{l-i}^{1} & \text{if }{j<l, i+j\geq 2l,}\\
\e_{l}+\e_{-l} - \e_{(i+j)-3l}^3-\e_i^2-\e_{l-j}^{2} & \text{if }{i<l, i+j\geq 2l.}
\end{cases}\end{align*}
We remind the readers of our convention that $\e_0^a$ is the zero vector, and $\e_{\pm l}^a=\e_{\pm l}$.
The projected weight configuration $\widehat{\br{\bs{\sigma}}}_{3l}$ can be easily checked using \eqref{eq:wtproj}.

To prepare for our next step, we relabel the vertex of $\widehat{\hexagon}_l$ according to the rule
$(i,j)\mapsto (l-i,l-j)$ (see Figure \ref{F:hathexagon} for an example).
Moreover, instead of the quiver with dimension vector $(TK_l,\beta_l')$, we consider its {\em reflection-equivalent} one.

When a vertex $s$ is a sink or source for a quiver $Q$ (without potential),
we have the {\em BGP-reflection functor} $\mu_s$ (see \cite{Ka1}).
At the level of $K_0$-groups, the functor $\mu_s$ sends a dimension vector $\beta$ to
$\mu_s\beta = \beta - (\innerprod{\beta,\e_s}_Q+\innerprod{\e_s,\beta}_Q) \e_s$.
This induces an action on the weight vectors, sending a weight $\sigma$ to $\mu_s \sigma$ satisfying \begin{equation} \label{eq:refwt}
(\mu_s\sigma)(s) = -\sigma(s) \text{ and } (\mu_s\sigma)(i) = \sigma(i)+ a_{s,i}\sigma(s),
\end{equation}
where $a_{s,i}$ is the number of arrows from $s$ to $i$.
The semi-invariant ring $\SI_\beta(Q)$ is invariant under the BGP-reflection functors.
More precisely,
\begin{lemma}[\cite{Ka1}] \label{L:relection} Let $\mu_s$ be the reflection at a sink or source $s$ such that $(\mu_s \beta)(s) >0$, then
$\SI_\beta(Q) = \SI_{\mu_s \beta}(\mu_s Q)$ and $\SI_\beta(Q)_\sigma = \SI_{\mu_s \beta}(\mu_s Q)_{\mu_s\sigma}$.
\end{lemma}

For each positive arm, we apply a sequence of reflections (from right to left) $\bs{\mu}:=\bs{\mu}_1 \bs{\mu}_{2}\cdots\bs{\mu}_{l-1}$, where $\bs{\mu}_{{k}}:=\mu_k\cdots\mu_2\mu_{1}$.
We get the following quiver $K_{l,l}^{2,3}$ with the same dimension vector (denoted by $\beta_l^3$).
$$\tripleflagged$$

\noindent It is easy to compute from \eqref{eq:refwt} that this sequence of reflection acts on weights by
\begin{equation} \label{eq:Refwt}
(\bs{\mu}\sigma)(l) = \sum_{i=1}^l\sigma({i}), \text{ and } (\bs{\mu}\sigma)({i}) = -\sigma(l-i).
\end{equation}
After this sequence of reflection, the weight configuration $\widehat{\br{\bs{\sigma}}}_{3l}$ transforms into a new one denoted by $\widehat{\bs{\sigma}}_{l}$.
It can be verified by \eqref{eq:Refwt} that
\begin{align*}
\widehat{\bs{\sigma}}_l(i,j) = \begin{cases} \e_i^1 + \e_j^2 + \e_{l-i-j}^3 - \e_{l} & \text{for $i\geq 0,j\geq 0$,} \\
-(\e_j^1 + \e_i^2 + \e_{-l-i-j}^3 - \e_{-l}) & \text{for $i\leq 0,j\leq 0$,} \\
\e_i^1- \e_{j}^{1} \pm \e_{\pm l-i-j}^3  \mp \e_{\pm l}  & \text{for $i>0,j<0$,} \\
\e_j^2- \e_{i}^{2} \pm \e_{\pm l-i-j}^3  \mp \e_{\pm l}  & \text{for $i<0,j>0$.}
\end{cases} \end{align*}
Here, $\pm$ is the sign of $i+j$. We keep the convention that $\e_0^a=0$ and $\e_{\pm l}^a=\e_{\pm l}$.

We summarize our result in step two.
\begin{proposition} The semi-invariant ring $\SI_{\beta_{l}^3}(K_{l,l}^{2,3})$ is the graded upper cluster algebra $\br{\mc{C}}(\widehat{\hexagon}_l; \widehat{\bs{\sigma}}_{l})$.
Moreover, $(\widehat{\hexagon}_l,W_{\widehat{\hexagon}_l})$ is a cluster model.
\end{proposition}
\noindent The potential $W_{\widehat{\hexagon}_l}$ is obtained from $W_n$ by deleting a few vertices,
so it is given by the sum of all ``smallest" triangle cycles.
Similar to the proof of \cite[Theorem 6.8]{Fs1} and Theorem \ref{T:LP_Gl},
we can show that $G(\widehat{\hexagon}_l,W_{\widehat{\hexagon}_l})$ is given by the lattice points inside a rational polyhedral cone $\mr{G}_{\widehat{\hexagon}_l}$.
The cone $\mr{G}_{\widehat{\hexagon}_l}$ has a following description:
$\h\in \mb{R}^{(\widehat{\hexagon}_l)_0}$ is in $\mr{G}_{\widehat{\hexagon}_l}$ if and only if
\begin{align} \label{eq:h0}
&\h(-l,l)\geq 0, && \sum_{n=0}^{m} \h(l-n,-l+n)\geq 0, &&\text{ for $m< l$};\\
&\sum_{n=j}^{m} \h(j,l-n) \geq 0, && \sum_{n=j}^{m}\h(n-l,-j) \geq 0, &&\text{ for $m< 2l$};\\
&\sum_{n=0}^{m} \h(k-n,-l+n) \geq 0, && \sum_{n=0}^{m}\h(l-n,n-k) \geq 0, &&\text{ for $m< l+k$};\\
\label{eq:h1} &\sum_{n=0}^{m} \h(-j,l-n) \geq 0, && \sum_{n=0}^{m}\h(n-l,j) \geq 0, &&\text{ for $m< 2l-j$}.
\end{align}

\subsection*{Step 3}
The quiver $K_{l,l}^2$ is a subquiver of $K_{l,l}^{2,3}$.
We make a convention that the arms of $K_{l,l}^2$ is identified with the arms of $K_{l,l}^{2,3}$ labeled with $(1)$.
Then a weight $\sigma$ of $K_{l,l}^2$ can be extended by zeros to a weight of $K_{l,l}^{2,3}$ still denoted by $\sigma$.
The semi-invariant ring $\SI_\bl(K_{l,l}^2)$ can be naturally embedded into $\SI_\bl(K_{l,l}^{2,3})$:
$$\SI_\bl(K_{l,l}^2) \cong \bigoplus_{\sigma} \SI_\bl(K_{l,l}^{2,3})_\sigma \hookrightarrow \SI_\bl(K_{l,l}^{2,3}),$$
where $\sigma$ runs through all weights supported on $K_{l,l}^2$.

To finish the proof, we need to show that there are the same number of lattice points in the polytopes $\mr{G}_{\Diamond_l}(\sigma)$ and $\mr{G}_{\widehat{\hexagon}_l}(\sigma)$.
For this, let us recall that a matrix $T$ is called {\em totally unimodular}
if all its minors equal to $0,1$, or $-1$.
This implies that a totally unimodular matrix maps lattice points to lattice points.
If a nonsingular square matrix is totally unimodular, then by Cramer's rule its inverse is also totally unimodular.
Let $T$ be an $m\times n$ totally unimodular matrix with $m\leq n$.
If $T$ is of full rank, then it has a totally unimodular right inverse as follows.
We take a full rank square submatrix $S$ of $T$. Without loss of generality, we can assume $S$ to be first $m$ columns of $T$.
We extend the inverse of $S$ by zeros to an $n\times m$ matrix, which is our desired right inverse.
We recall Ghouila-Houri's characterization of totally unimodular matrix.
\begin{lemma}[{\cite[Theorem 19.3.(iv)]{Sc}}] \label{L:TU} A matrix $T$ is totally unimodular if and only if each collection of columns of $T$ can be split into two parts so that the sum of the columns in one part minus the sum of the columns in the other part is a vector with entries only $0,1$, and $-1$.
\end{lemma}

\begin{proposition} There is a totally unimodular full-rank~linear~map~$\mb{R}^{\Diamond_l}\!\to\! \mb{R}^{\widehat{\hexagon}_l}$
mapping the polytope $\mr{G}_{\Diamond_l}(\sigma)$ onto the polytope $\mr{G}_{\widehat{\hexagon}_l}(\sigma )$.
In particular, the two polytopes have the same number of lattice points.
\end{proposition}

\begin{proof} We define the linear map $\varphi:\mb{R}^{\Diamond_l} \to \mb{R}^{\widehat{\hexagon}_l}$ as follows
\begin{align*}
\varphi(\e_{l;l,0})&= \e_{-l,l}, \\
\varphi(\e_{i;j,k})&= \e_{k,-i}+\e_{-j,j} +\e_{j,0} - \e_{-j,0} - \e_{0,j}  && \text{for $j\neq l$,}\\
\varphi(\e_{j,k;i})&= \e_{i,-k}+\e_{-j,j} +\e_{0,-j} - \e_{-j,0} - \e_{0,j}  && \text{for $j\neq l$.}
\end{align*}
We use the convention that $\e_{0,0}$ is the zero vector, so in particular $\varphi(\e_{i;0,i})= \e_{-i,i}$.
$\varphi$ is clearly of full rank, so we remain to check that
$$ \text{(1). $\varphi$ is totally unimodular; \qquad (2). $\varphi$ maps $\mr{G}_{\Diamond_l}(\sigma)$ onto $\mr{G}_{\widehat{\hexagon}_l}(\sigma)$.} $$

(1). To verify the total unimodularity, we observe that the matrix of $\varphi$ can be partition into blocks indexed by $j$ and a sign $\pm$.
If $j\neq 0,l$, the rows and columns of each block $B_j^\pm$ are indexed by
$\{(i;j,k)\}_k \times \{(k,-i),(-j,j),(j,0),(-j,0),(0,j)\}_k$ or $\{(j,k;i)\}_k \times \{(i,-k),(-j,j),(0,-j),(-j,0),(0,j)\}_k$.
If $j=0$ or $l$, then the block has a single entry $1$.
Since no any two blocks have common rows or columns, it suffices to verify the total unimodularity on each block. But this is rather trivial by Lemma \ref{L:TU}.

(2) We first verify that $\g$ and $\varphi(\g)$ have the same $\sigma$-weight.
\begin{align*}
\e_{i;j,k} \bs{\sigma}_l &= \e_j + \e_k - \e_{-i}, \text{ while}\\
\varphi(\e_{i;j,k}) \widehat{\bs{\sigma}}_l
&=\widehat{\bs{\sigma}}_l(k,-i)+\widehat{\bs{\sigma}}_l(-j,j)+\widehat{\bs{\sigma}}_l(j,0)
-\widehat{\bs{\sigma}}_l(-j,0) -\widehat{\bs{\sigma}}_l(0,j)\\
&= (\e_{k}^{1} - \e_{-i}^1 -\e_{-l-k+i}^3 +\e_{-l})
+ (\e_j^2 - \e_{-j}^{2})
+(\e_j^1 + \e_{l-j}^3 - \e_{l}) \\
&\quad\ +(\e_{-j}^2 + \e_{-l+j}^3 - \e_{-l}) -(\e_j^2 + \e_{l-j}^3 - \e_{l}) \\
&= \e_{j}^{1} +\e_{k}^1 - \e_{-i}^1.
\end{align*}
Similarly we can verify that $\varphi(\e_{j,k;i})=\e_i^1 - \e_{-j}^{1} -\e_{-k}^1$.

Next, we write the definition of $\varphi$ in coordinates: $\h=\varphi(\g)$.
We find that $\h(i,j)=0$ except for the following
\begin{align*}
&\h(-l,l)= \g(l;l,0); \\
&\h(k,-i)= \g(i;j,k),\quad \h(i,-k)= \g(j,k;i) \quad \text{for } i\geq k,k\neq 0;\\
&\h(j,0)= \sum_{i-k=j} \g(i;j,k),\quad \h(0,-j)= \sum_{i-k=j} \g(j,k;i); \\
&\h(-j,j) = -\h(0,j)=-\h(-j,0) = \sum_{i-k=j} \big(\g(i;j,k) + \g(j,k;i)\big) - \g(j;j,0).
\end{align*}
We see that $\h(j,0)$ (resp. $\h(0,-j)$) is a dimension vector of a subrepresentation of $T_{i;j,k}$ (resp. $T_{j,k;i}$),
and $\h(-j,j)$ is a sum of dimension vectors of a subrepresentation of $T_{i;j,k}$ and a subrepresentation of $T_{j,k;i}$.
Now we can easily verify the following
\begin{align*}
&\sum_{n=0}^{m} \h(l-n,-l+n) = \sum_{n=0}^{m} \g(l-n;0,l-n), \\ 
&\sum_{n=j}^{m} \h(j,l-n) = \delta_1 \sum_{i-k=j} \g(i;j,k) + \sum_{n=1}^{\min(m-l,j)} \g(j-n,n;j) + \sum_{n=1}^{m-l-j} \g(j+n;n,j), \\ 
&\sum_{n=0}^{m} \h(k-n,-l+n) = \sum_{n=0}^{\min(m,k-1)} \g(l-n;l-k,k-n)+\delta_2 \h(0,-l+k) -\delta_1 \h(-l+k,l-k), \\ 
&\sum_{n=0}^{m} \h(-j,l-n) = \delta_3\h(-j,j), 
\end{align*}
where $\delta_1=1$ if $m\geq l$, $\delta_2=1$ if $m\geq k$, $\delta_3=1$ if $l-j \leq m< l$; otherwise $\delta_i=0$ for $i=1,2,3$.
We can rewrite the first three equations as follows:
\begin{align*}
\sum_{n=0}^{m} \h(l-n,-l+n) &= \g \cdot \gamma_{0,m}  &&\text{ for $m\leq l$},\\
\sum_{n=j}^{m} \h(j,l-n) &= \g \cdot \gamma_{1,m}  &&\text{ for $m\leq 2l$},\\
\sum_{n=0}^{m} \h(l-n,-k+n) &= \g \cdot (\gamma_{2,m} + \delta_2 \h(0,-l+k)) &&\text{ for $m\leq l+k$},
\end{align*}
where $\gamma_{0,m}$ is a dimension vector of some submodule of $T_{l;0,l}$,
$\gamma_{1,m}$ and $\gamma_{2,m}$ are dimension vectors of some submodules of $T_{i;j,k}$.
Note that the negative term of the third equation causes cancelation only.
We observe that all dimension vectors of submodules of $T_{l;0,l}$ and $T_{i;j,k}$ arise this way.
By symmetry, we have the similar equations for $\sum_{n=j}^{m} \h(n-l,-j), \sum_{n=0}^{m} \h(l-n,n-k)$, and $\sum_{n=0}^{m} \h(n-l,j)$.
We conclude that $\h$ satisfies the defining condition \eqref{eq:h0}--\eqref{eq:h1} of $\mr{G}_{\widehat{\hexagon}_l}$ if and only if $\g$ satisfy the defining conditions of $\mr{G}_{\Diamond_l}$.
This implies that $\varphi$ maps $\mr{G}_{\Diamond_l}(\sigma)$ (bijectively) onto $\mr{G}_{\widehat{\hexagon}_l}(\sigma)$.
Therefore, the two polytopes have the same number of lattice points.
\end{proof}

\bibliographystyle{amsplain}

\end{document}